\documentclass[a4paper,twoside]{article}

\usepackage[english]{babel}
\usepackage[T1]{fontenc}
\usepackage{lmodern}

\usepackage{csquotes}

\usepackage{mathtools, amsmath, amssymb, amsthm, colonequals, cases}
\usepackage[titletoc,title]{appendix}
\usepackage{fullpage, microtype}
\usepackage{hyperref}
\usepackage[nameinlink]{cleveref}
\usepackage{fancyhdr}

\title{On the Invertibility of the Potential-Density Mapping for the Vlasov--Poisson System in Analytic Spaces}
\author{Simon Le Bouëdec\footnote{Université Paris-Saclay, CNRS, Laboratoire de Mathématiques d'Orsay, 91405 Orsay CEDEX, France, \\ Email address: simon.le-bouedec@universite-paris-saclay.fr}}

\date{}

\AtBeginEnvironment{appendices}{\crefalias{section}{appendix}}

\newtheorem{theorem}{Theorem}
\newtheorem{lemma}[theorem]{Lemma}
\newtheorem{proposition}[theorem]{Proposition}

\newtheorem{claim}{Claim}
\theoremstyle{definition}
\newtheorem{definition}[theorem]{Definition}
\newtheorem{remark}[theorem]{Remark}

\Crefname{claim}{Claim}{Claims}

\DeclareMathOperator{\ext}{\mathrm{ext}}
\DeclareMathOperator{\dd}{\mathrm{d}\!}

\newcommand{\rom}[1]{\textup{\uppercase\expandafter{\romannumeral#1}}}
\newcommand{\I}{\rom{1}}
\newcommand{\II}{\rom{2}}
\newcommand{\III}{\rom{3}}

\DeclarePairedDelimiter\abs{\lvert}{\rvert}
\DeclarePairedDelimiter\norm{\lVert}{\rVert}
\newcommand{\normop}[1]{{\left\vert\kern-0.25ex\left\vert\kern-0.25ex\left\vert #1\right\vert\kern-0.25ex\right\vert\kern-0.25ex\right\vert}}
\DeclarePairedDelimiterX{\jap}[1]{\langle}{\rangle}{#1}
\DeclarePairedDelimiterX{\inp}[2]{\langle}{\rangle}{#1, #2}

\usepackage{scalerel,stackengine}
\stackMath
\newcommand\ultrawidehat[1]{%
\savestack{\tmpbox}{\stretchto{%
  \scaleto{%
    \scalerel*[\widthof{\ensuremath{#1}}]{\kern-.6pt\bigwedge\kern-.6pt}%
    {\rule[-\textheight/2]{1ex}{\textheight}}
  }{\textheight}%
}{0.5ex}}%
\stackon[1pt]{#1}{\tmpbox}%
}
\allowdisplaybreaks

\begin{document}

\maketitle

\begin{abstract}
    We develop the mathematical analysis of the Time-Dependent Density Functional Theory introduced by Manfredi [\textit{J. Plasma Phys.}, 86.2 (2020), 825860201] in the setting of Vlasov equations with a time-dependent exterior potential. In this context, we establish the invertibility of the potential-density mapping for the Vlasov--Poisson system on \(\mathbb{T}^d \times \mathbb{R}^d\) when the exterior potential is (real-)analytic in the space variable. We also recover a Runge--Gross type Theorem for the Vlasov--Poisson equation for smooth (real-)analytic in time exterior potential.

    The surjectivity of the potential-density mapping is an inverse problem for the Vlasov--Poisson equation: we find the electric field that generates a given density profile on a given time interval.
\end{abstract}

\tableofcontents

\section{Introduction}

In this paper, we study the potential-density mapping for the Vlasov--Poisson system describing the dynamics of a collisionless plasma. This mean-field equation describes the evolution of the distribution \(f(t, x, v)\) of electrons in phase-space \([0, T]\times\mathbb{T}^d \times \mathbb{R}^d\), for \(T>0\) and \(d \in \mathbb{N}^*\), which satisfies
\begin{equation}\label{eq_vlasov-poisson}
    \begin{cases}
        \partial_t f + v \cdot \nabla_x f + \nabla_x V \cdot \nabla_v f+\nabla_x V^{\ext} \cdot \nabla_v f = 0, \\
        -\Delta_x V = \int_{\mathbb{R}^d} f \dd v - 1,                                                          \\
        f|_{t=0} = f^0,
    \end{cases}
\end{equation}
where the forces applied to the system are generated by the self-induced interaction potential \(V(t, x)\) and the given time-dependent exterior potential \(V^{\ext}(t, x)\). The initial condition, \(f^0(x, v)\), is fixed throughout this paper and can be taken nonnegative satisfying
\[
    \int_{\mathbb{T}^d}\int_{\mathbb{R}^d}f^0(x, v)\dd x\dd v=1.
\]
We also define the spatial density of \(f\) by \(\rho[f](t, x) \colonequals \int_{\mathbb{R}^d} f(t, x, v) \dd v\).

We wish to determine whether the spatial density of particles fully characterizes the exterior potential on a given time interval, i.e., whether the so-called potential-density mapping,
\[
    \nabla_x V^{\ext} \longmapsto \rho[f]\qquad\text{for a given time interval},
\]
specifically identifying the functional spaces for which this mapping is bijective. It is natural to consider \(\nabla_xV^{\ext}\) rather than \(V^{\ext}\), since only the former appears in the equation: offsetting the exterior potential by a constant, even a time-dependent one, yields the same equation. Moreover, since \(f\) is completely determined by \(\nabla_x V^{\ext}\), we shall write \(\rho[\nabla_x V^{\ext}]\) instead of \(\rho[f]\), and write the mapping as \(\nabla_x V^{\ext}\mapsto\rho[\nabla_x V^{\ext}]\). We will study this application in two settings.
\begin{description}
    \item[Space analyticity] Let \(\lambda, s\geqslant 0\). For \(h \in L^2(\mathbb{T}^d)\), set \(\norm{h}_{\lambda, s} \colonequals \norm{e^{\lambda \jap{\nabla}}h}_{H^s(\mathbb{T}^d)}\).
        We denote by \(\mathcal{A}^{\lambda, s}\subset C^{\infty}(\mathbb{T}^d)\) the functional space associated with this norm. As explained before, we consider non-offset exterior potentials, so we will work with the space
        \[
            \nabla L_T^{\infty}\mathcal{A}^{\lambda, s} \colonequals\{\nabla V^{\ext};~V^{\ext}\in L_T^{\infty}\mathcal{A}^{\lambda, s}\}.
        \]

    \item[Time analyticity] We also define the original regularity class in which this problem was first studied,
        \[
            \nabla\mathcal{A}_T\colonequals\Biggl\{\nabla_xV^{\ext};~\begin{aligned}
                 & V^{\ext}\in C^{\infty}([0, T]\times\mathbb{T}^d),                      \\
                 & \forall x\in\mathbb{T}^d,t\mapsto V^{\ext}(t, x)~\text{real-analytic}.
            \end{aligned}\Biggr\}.
        \]
\end{description}
For further details on the notations see the next \cref{sec_notations}.

Succinctly, we will prove the injectivity and surjectivity of the potential-density mapping in the following case:
\[
    \nabla_xV^{\ext}\in \nabla L^{\infty}_{T}\mathcal{A}^{\lambda,s}\longmapsto\rho[\nabla_x V^{\ext}]\in C([0, T'];~\mathcal{A}^{\lambda',s}),
\]
and also, its injectivity in the case:
\[
    \nabla_xV^{\ext} \in \nabla\mathcal{A}_T \longmapsto \rho[\nabla_x V^{\ext}] \in C^{\infty}([0,T]\times\mathbb{T}^d).
\]
The precise statements of these results are given in \cref{sec_main_results}.

We wish to anticipate and already highlight that the Coulomb potential \(V\) given by the Poisson equation does not play a crucial role in yielding the injectivity or surjectivity: it can be replaced by any reasonable self-interaction potential of the form \(V(x) = \int k(x, y)\rho(y)\dd y\).

\subsection{The Time-Dependent Density Functional Theory context}\label{sec_tddft_context}

In quantum mechanics, during the 1960s, Density Functional Theory (DFT) was developed and became a powerful tool to study ground states of systems in physics, chemistry, material science, and biology. Introduced by Hohenberg and Kohn in \cite{PhysRev.136.B864}, it allows reducing a many-body problem to an equivalent one-body problem---in the sense that they generate the same particle density---which is much easier to solve.

Naturally, the question of extending DFT to the time-dependent setting arose, and it was answered by Runge and Gross in the foundational paper of Time-Dependent Density Functional Theory (TDDFT) \cite{PhysRevLett.52.997} in 1984. TDDFT, like DFT, rests on establishing that the density of particles is fully characterized by the exterior potential. Let us briefly describe the principle of TDDFT in its original setting, namely the Schrödinger equation. We refer to \cite{Ruggenthaler_2015, PhysRevA.93.062510} and references therein for a more detailed presentation of the quantum case. Consider an \(N\)-body quantum system described by the Schrödinger equation
\[
    \begin{cases}
        i\partial_t\psi(t, x) = H \psi(t, x), \qquad x\in\mathbb{R}^{3N}, \\
        \psi|_{t=0} = \psi^0,
    \end{cases}
\]
where the Hamiltonian \(H\) is given by
\[
    H \colonequals \sum_{j=1}^N \left(-\frac{1}{2}\Delta_{x_j} + V^{\ext}(t, x_j)\right) + \sum_{1\leqslant j < k \leqslant N} w(x_j - x_k),
\]
where \(w\) is a particle-particle interaction potential, and \(V^{\ext}\) is an exterior potential. To show that the density of particles defined by
\[
    \rho(t, x) = N\int_{\mathbb{R}^{3(N-1)}}\abs{\psi(t, x, x_2, \ldots, x_N)}^2\dd x_2 \cdots \dd x_N,\qquad x\in\mathbb{R}^3,
\] can be generated equivalently by \(N\) one-body equations of the form
\[
    \begin{cases}
        i\partial_t\psi_j(t, x) = \left(-\frac{1}{2}\Delta + V^{\ext, \mathrm{s}}(t, x)\right) \psi_j(t, x), \qquad x\in\mathbb{R}^3, \\
        \psi_j|_{t=0} = \psi_j^0,
    \end{cases}
\]
where \(V^{\ext, \mathrm{s}}\) is a suitable effective potential, with density \(\rho_{\mathrm{s}}(t, x)=\sum_{j=1}^N \abs{\psi_j(t, x)}^2\); we need to establish the invertibility of the potential-density mapping
\[
    V^{\ext} \longmapsto \rho_{\mathrm{s}}[V^{\ext}],\qquad V^{\ext, \mathrm{s}} \longmapsto \rho_{\mathrm{s}}[V^{\ext, \mathrm{s}}].
\]
Indeed, if we show that the density \(\rho\) is in the image of the mapping, we can find an effective potential \(V^{\ext, \mathrm{s}}\)---see the Kohn--Sham procedure \cite{PhysRev.140.A1133} to get an explicit expression---such that the density generated by the one-body equations \(\rho_{\mathrm{s}}\) is exactly the \(N\)-body particle density \(\rho\).

This is particularly relevant for numerical analysis, as solving \(N\) one-body equations is much more tractable than solving a single \(N\)-body equation.

Recently, interest in applying the principle of TDDFT to the classical Vlasov--Poisson system has emerged, spearheaded by the work of Manfredi \cite{Manfredi_2020} in 2020. The goal is to reduce the Vlasov--Poisson equation \eqref{eq_vlasov-poisson} to the following fluid equation
\[
    \begin{cases}
        \partial_t \rho + \nabla_x \cdot J = 0, \\
        \partial_t J + \nabla_x \cdot K[\rho] = \rho\nabla_x V[\rho] + \rho\nabla_x V^{\ext}[\rho],
    \end{cases}
\]
where \(\rho\) is the spatial density, \(J\) the current, and \(K\) the kinetic energy density. The potential \(V[\rho]\) only depends on \(\rho\) by the Poisson equation. For the equation to be well-defined, \(V^{\ext}[\rho]\) and \(K[\rho]\) need to be fully characterized by the density, i.e., the potential-density mapping needs to be invertible. Since this is a new development, the mathematical framework has yet to be established, and the question of the invertibility of the potential-density mapping for the Vlasov--Poisson equation is the focus of this paper.

Similarly to the quantum case and underlined in \cite{Manfredi_2020}, we get computational advantages by reducing the Vlasov--Poisson equation to a fluid equation. It was already known that certain kinetic models and behaviors, such as Landau damping, can be captured by fluid equations, see \cite{BERTRAND196868,PhysRevLett.64.3019}. Nonetheless, the insight from \cite{Manfredi_2020} on TDDFT for the Vlasov--Poisson system, is that it yields a systematic approach---which we shall detail later, cf. \Cref{thm_surjectivity}---to derive an equivalent fluid equation and could account for more complex kinetic behaviors, which motivates this work.

\subsection{Motivations for the analytic framework}

As previously stated, in this work, we study the invertibility of the potential-density mapping for the Vlasov--Poisson system in analytic spaces. We wish to motivate this framework here.

We start by giving a short review of the well-posedness of the Vlasov--Poisson equation. It has been established in several works: local and global existence in physical dimension \(d\leqslant 3\), see \cite{MR794002,MR1126425,MR1115549,MR1165424,MR1132787}. For this paper, the recent work in \cite{MR4266238}---which was motivated by the earlier works \cite{MR3489904,MR2863910,MR1056129}---is particularly relevant, as it establishes local existence results in the analytic spaces we consider but also global existence in the case \(d=3\).

The time analytic framework is the historic setting in which the foundational result of TDDFT for the Schrödinger equation, the Runge--Gross theorem \cite{PhysRevLett.52.997}, was proved. In the quantum case, it was shown to fail in the Coulomb interaction case in \cite{PhysRevA.93.062510}, but it still holds true for the Vlasov--Poisson system, as we will show in \Cref{thm_injectivity_time_analytic}. To get this result, the propagation of moments in velocity is an important technical ingredient This property is established in the works previously cited, as well as in \cite{MR3842911}.

Finally, let us explain the need for the space-velocity analytic framework. It is not merely a technical reason due to the loss of derivative but a fundamental one. It pertains to the surjectivity of the potential-density mapping. We  can formulate this problem as finding a suitable exterior potential \(V^{\ext}\) such that the density of particles generated by the Vlasov--Poisson equation is a target density \(\varrho\). This means, \(\nabla_x V^{\ext}\), which was previously given now becomes an unknown: the Vlasov equation, with unknowns \((f, V^{\ext})\), then reads
\begin{equation}\label{eq_vlasov-surjectivity}
    \begin{cases}
        \partial_t f + v \cdot \nabla_x f + \nabla_x V^{\ext} \cdot \nabla_v f = 0, \\
        \int_{\mathbb{R}^d} f(t,x,v) \dd v = \varrho(t,x),                          \\
        f|_{t=0} = f^0,
    \end{cases}
\end{equation}
where we dropped the Coulomb potential so that the ensuing argument becomes clearer. Now, under the lens of an inverse problem, we show---albeit in a particular case---that this problem is ill-posed in finite regularity, which is precisely the obstruction that requires the analytic framework to overcome and yield the surjectivity of the potential-density mapping.

Indeed, in the case of a constant target density, \(\varrho(t, x) \equiv 1\), we recognize the kinetic incompressible Euler equation named by Brenier in \cite{Brenier1989,MR1748352}. In this very case, this equation has been shown to be ill-posed in weighted Sobolev spaces, see \cite{MR3026566,MR3509003,MR4093619}.

We refer to \Cref{prop_illposedness} for a precise statement on the obstruction to surjectivity in finite regularity.

\subsection{Notations and definitions}\label{sec_notations}

Let us start by introducing some basic notations. We recall the standard multi-index notation, for \(\alpha = (\alpha_1, \ldots, \alpha_d) \in \mathbb{N}^d\), we write
\[
    \abs{\alpha} \colonequals \sum_{j=1}^d \alpha_j, \qquad \partial_x^\alpha \colonequals \partial_{x_1}^{\alpha_1} \cdots \partial_{x_d}^{\alpha_d}.
\]
We denote by \(\mathbb{T}^d = (\mathbb{R}/2\pi\mathbb{Z})^d\) the \(d\)-dimensional torus. For a function \(g = g(x, v)\) with \((x, v) \in \mathbb{T}^d \times \mathbb{R}^d\), its unitary Fourier transform is defined by
\begin{align*}
    \hat{g}_k(\eta) & \colonequals (2\pi)^{-d} \int_{\mathbb{T}^d \times \mathbb{R}^d} e^{-ix \cdot k - iv \cdot \eta} g(x, v) \dd x \dd v, \\
    g(x, v)         & = (2\pi)^{-d} \sum_{k \in \mathbb{Z}^d} \int_{\mathbb{R}^d} e^{ix \cdot k + iv \cdot \eta} \hat{g}_k(\eta) \dd\eta,
\end{align*} where \(k \in \mathbb{Z}^d\) and \(\eta \in \mathbb{R}^d\). Analogous conventions are used for functions of \(x\) or \(v\) alone. With these conventions, the following relations hold
\[
    \begin{cases}
        \int_{\mathbb{T}^d \times \mathbb{R}^d} g(x, v) \overline{h(x, v)} \dd x \dd v = \sum_{k \in \mathbb{Z}^d} \int_{\mathbb{R}^d} \hat{g}_k(\eta) \overline{\hat{h}_k(\eta)} \dd\eta, \qquad \ultrawidehat{g h} = (2\pi)^{-d} \hat{g} * \hat{h}, \\
        (\ultrawidehat{\nabla g})_k(\eta) = (ik, i\eta) \hat{g}_k(\eta), \qquad (\ultrawidehat{v^\alpha g})_k(\eta) = D_\eta^\alpha \hat{g}_k(\eta).
    \end{cases}
\]
We use the Japanese bracket notation for \(x, y \in \mathbb{R}^d\),
\[
    \jap{x} \colonequals \sqrt{1 + \abs{x}^2}, \qquad \inp{x}{y} \colonequals \sqrt{1 + \abs{x}^2 + \abs{y}^2}.
\]
We define the following family of functional analytic spaces which are at the core of our analysis.
\begin{definition}[Analytic Norms]
    Let \(M, s \geqslant 0\) and \(\lambda \geqslant 0\) (the analyticity radius). For \(f \in L^2(\mathbb{T}^d \times \mathbb{R}^d)\), define
    \[
        \begin{cases}\norm{f}_{\lambda, s}^2 &\colonequals \sum_{k \in \mathbb{Z}^d} \int_{\mathbb{R}^d} \abs{\hat{f}_k(\eta)}^2 \inp{k}{\eta}^{2s} e^{2\lambda \inp{k}{\eta}} \dd\eta = \norm{A^{(\lambda, s)} \hat{f}}_{\ell^2_k L^2_\eta}^2, \\ \norm{f}_{M, \lambda, s}^2 &\colonequals \sum_{\abs{\alpha} \leqslant M} \norm{v^\alpha f}_{\lambda, s}^2, \end{cases}
    \]
    where \(A^{(\lambda, s)}_k(\eta) \colonequals \inp{k}{\eta}^s e^{\lambda \inp{k}{\eta}}\) is a Fourier multiplier.

    For \(\rho \in L^2(\mathbb{T}^d)\), set
    \[
        \norm{\rho}_{\lambda, s}^2 \colonequals \sum_{k \in \mathbb{Z}^d} \abs{\hat{\rho}_k}^2 \jap{k}^{2s} e^{2\lambda \jap{k}} = \norm{B^{(\lambda, s)} \hat{\rho}}_{\ell^2_k}^2,
    \]
    where \(B^{(\lambda, s)}_k \colonequals \jap{k}^s e^{\lambda \jap{k}}\) is a Fourier multiplier.

    We denote by \(\mathcal{A}^{\lambda, s}\subset C^{\infty}(\mathbb{T}^d)\) the functional space associated with this norm.

    Finally, we define the adequate function spaces with compatibility conditions for exterior potentials \(V^{\ext}\) and spatial densities \(\rho\) by defining
    \[
        \mathfrak{V}       \colonequals \left\{V^{\ext}=V^{\ext}(t, x)\left|~\begin{aligned}
             & \exists T, \lambda, s>0, V^{\ext} \in L^{\infty}([0, T]; \mathcal{A}^{\lambda, s}),           \\
             & \exists\rho_{\min}>0,\rho[V^{\ext}]\geqslant\rho_{\min}>0~\text{on}~[0, T]\times\mathbb{T}^d.
        \end{aligned}\right.\right\}
    \]
    the set of analytic exterior potentials generating an elliptic density;
    \[
        \nabla\mathfrak{V} \colonequals \left\{\nabla_x V^{\ext};~V^{\ext}\in\mathfrak{V}\right\},
    \]
    the same set without (time-dependent) constants, and,
    \[
        \mathfrak{D}  \colonequals \left\{\left(\rho, f^0\right) =\left(\rho(t,x), f^0(x, v)\right)\left|~\begin{aligned}
             & \textstyle\exists T, \lambda, s>0, \rho \in C^2([0, T]; \mathcal{A}^{\lambda, s}),~\exists \lambda', s'>0, f^0 \in \mathcal{A}^{\lambda', s'}, \\
             & \textstyle\rho(0, x)=\int_{\mathbb{R}^d}f^0(x, v)\dd v,~\partial_t\rho(0, x)=-\nabla_x\int_{\mathbb{R}^d} v f^0(x,v) \dd v,                    \\
             & \textstyle\exists\rho_{\min}>0,\rho(t, x)\geqslant\rho_{\min}>0~\text{on}~[0, T]\times\mathbb{T}^d.
        \end{aligned}\right.\right\}
    \]
    the set of density-initial datum compatible pairs.
\end{definition}

We choose these analytic norms because they are well‑behaved---i.e., they share the same properties as Sobolev norms with respect to inequalities, estimates, and regularity---and well adapted to the Vlasov--Poisson equation as shown in \cite{MR2863910,MR3489904,MR4266238}. Our main interest lies in the analyticity radius \(\lambda\), which allows for a precise quantitative control of the regularity of the solutions. Precisely, by making \(\lambda\) time-dependent (\(t\mapsto\lambda(t)\)), we can gain regularity on our estimates which balances derivative losses appearing in the analysis of the potential-density mapping. This is central to our proofs.

\section{Main results}\label{sec_main_results}

\subsection{Results}

We state two results on the injectivity of the potential-density mapping
\[
    \nabla_x V^{\ext}\longmapsto \left(\rho[f], f|_{t=0}\right)
\]
for the Vlasov--Poisson system.

The first theorem assumes real-analytic regularity in the space-velocity variables for the solutions and the exterior potentials.
\begin{theorem}\label{thm_injectivity_space_analytic}
    Let \(V^{\ext}_j\) \((j=1,2)\) be two exterior potentials and let \(f_j\) be the corresponding solution of the Vlasov–Poisson equation \eqref{eq_vlasov-poisson} with spatial density \(\rho_j\) and fixed initial datum \(f^0\).

    Assume that for some \(T>0\), \(M>d/2+2\), \(\lambda_0>0\) and \(s\geqslant d+2\) there exists \(C>0\) such that for \(j=1,2\)
    \[
        \sup_{t\in[0,T]}\norm{f_j(t)}_{M,\lambda_0,s} \leqslant C, \qquad \sup_{t\in[0,T]}\norm{V^{\ext}_j(t)}_{\lambda_0,s} \leqslant C .
    \]
    If \(\rho_1(t,x)=\rho_2(t,x)\equalscolon\rho(t,x)\) for all \((t,x)\in[0,T]\times\mathbb{T}^d\) and the common density satisfies \(\rho\geqslant\rho_{\min}>0\) on \([0,T]\times\mathbb{T}^d\), then,
    \[
        \nabla_x V^{\ext}_1(t,x)=\nabla_x V^{\ext}_2(t,x)\qquad\text{for all }(t,x)\in[0,T]\times\mathbb{T}^d .
    \]
    Simply put, the map
    \[
        \nabla_xV^{\ext}\in \nabla\mathfrak{V}\longmapsto \left(\rho[\nabla_xV^{\ext}], f|_{t=0}\right)\in \mathfrak{D}
    \]
    is injective.
\end{theorem}

\begin{remark}
    The analytic regularity hypothesis is satisfied, for instance, when the initial datum \(f^0\) and the exterior potential are analytic, see \cite{MR4266238}.
\end{remark}

The second theorem is a Vlasov--Poisson adaptation of the Runge--Gross theorem, see \cite{PhysRevLett.52.997}, which provided the foundational result for TDDFT. It replaces the spatial and velocity analyticity in \Cref{thm_injectivity_space_analytic} by analyticity in time, while still assuming \(C^{\infty}\) regularity in the spatial and velocity variables.
\begin{theorem}\label{thm_injectivity_time_analytic}
    Let \(V^{\ext}_j\) (\(j=1,2\)) be two exterior potentials and let \(f_j\) be the corresponding solution to the Vlasov--Poisson equation \eqref{eq_vlasov-poisson} with density \(\rho_j\) and fixed initial datum \(f^0\).

    Assume \(V^{\ext}_j\in C^{\infty}([0,T]\times\mathbb{T}^d)\) and \(f_j\in C^{\infty}([0,T]\times\mathbb{T}^d\times\mathbb{R}^d)\). Moreover, assume the uniform moment bounds
    \begin{equation}\label{hyp_thmITA_moments_bound}
        \forall k\in\mathbb{N},\qquad \abs{v}^{k+1}\abs{\nabla_x^k f_j(t,x,v)} \leqslant C_k(v),
    \end{equation}
    where each \(C_k(\cdot)\in L^1(\mathbb{R}^d_v)\) and the bounds hold uniformly in \((t,x)\).

    If \(\rho_1(t,x)=\rho_2(t,x)\eqqcolon\rho(t,x)\) for all \((t,x)\in[0,T)\times\mathbb{T}^d\) and the nodal set \(\{x\in\mathbb{T}^d\mid \rho(0,x)=0\}\) has zero volume, then for every \(k\in\mathbb{N}\) there exists a constant \(c_k\in\mathbb{R}\) such that
    \[
        \partial_t^k\bigl(V^{\ext}_1-V^{\ext}_2\bigr)(0,x)=c_k\qquad\text{for all }x\in\mathbb{T}^d.
    \]
    If, in addition, \(t\mapsto (V^{\ext}_1-V^{\ext}_2)(t,x)\) is real-analytic for each fixed \(x\), then
    \[
        V^{\ext}_1(t,x)-V^{\ext}_2(t,x)=C(t)=\sum_{k\ge0}\frac{c_k}{k!}\,t^k.
    \]
    Simply put, the map
    \[
        \nabla_xV^{\ext}\in\nabla\mathcal{A}_T\longmapsto \left(\rho[\nabla_xV^{\ext}], f|_{t=0}\right)\in C^{\infty}([0, T]\times\mathbb{T}^d)
    \]
    is injective.
\end{theorem}

\begin{remark}
    The moment-bound hypothesis \eqref{hyp_thmITA_moments_bound} on the solution is satisfied when the initial datum \(f^0\) also satisfies it, i.e., we use the propagation of moments in velocity for the Vlasov--Poisson system---see \cite{MR3842911} and \Cref{appendix_moments}.
\end{remark}

\begin{remark}
    Unlike the Schrödinger case with Coulomb interactions in the whole space---see \cite{PhysRevA.93.062510}, where the Runge--Gross type argument fails due to the singularity in \(0\)---the Runge--Gross Theorem still holds for the Vlasov--Poisson system.
\end{remark}

Now, we turn our attention to the surjectivity of the potential–density mapping. It should be understood as an inverse problem: from a measure of the density of some particles, we want to find the force, derived from an electric potential, to which these particles are subjected. It can also be interpreted as a control problem in the sense that we have a given target density, and we want to find an exterior potential that steers the system toward this target.
\begin{theorem}\label{thm_surjectivity}
    Let \(\varrho\) be a given spatial density such that the compatibility conditions \(\varrho(0, x)=\int_{\mathbb{R}^d}f^0(x, v)\dd v\) and \(\partial_t\varrho(0, x)=-\nabla_x\int_{\mathbb{R}^d} v f^0(x,v) \dd v\) are satisfied. Assume that for some \(T>0\), \(M>d/2+2\), \(\lambda_0>0\) and \(s\geqslant 1+d/2\),
    \begin{gather*}
        \norm{f^0}_{M, \lambda_0, s} < \infty, \qquad\sup_{t\in[0,T]}\norm{\varrho(t)}_{\lambda_0,s} < \infty, \\
        \sup_{t\in[0,T]}\norm{\partial_t^2\varrho(t)}_{\lambda_0,s} < \infty,\qquad\varrho(t,x) \geqslant \varrho_{\min}>0.
    \end{gather*}
    Then, for any \(\gamma \in (0, 1)\), there exist \(0<\lambda_0'=\lambda_0'(d, \varrho, f^0)\leqslant\lambda_0\), an exterior potential \(V^{\ext}\), \(c=c(d, s, \lambda_0, \varrho, \gamma, T)>\lambda_0\) and a \(C>0\) constant, such that for all \(0\leqslant\lambda<\lambda_0'\), and all \(0<t<\min\left\{T, \frac{\lambda_0'-\lambda}{c}\right\}\),
    \begin{align*}
        \norm{V^{\ext}(t)}_{\lambda, s}          & \leqslant C,                                       \\
        \norm{\nabla_x V^{\ext}(t)}_{\lambda, s} & \leqslant C (\lambda_0' - \lambda - ct)^{-\gamma},
    \end{align*}
    and the associated Vlasov--Poisson system \eqref{eq_vlasov-poisson} admits a solution \(f\) with spatial density \(\rho[f]=\varrho\).
    Simply put, the map
    \[
        \nabla_xV^{\ext}\in \nabla\mathfrak{V}\longmapsto \left(\rho[\nabla_xV^{\ext}], f|_{t=0}\right)\in \mathfrak{D}
    \]
    is surjective.
\end{theorem}

\begin{remark}
    It is thanks to this surjectivity that we can find an effective exterior potential as desired in the TDDFT context described in \cref{sec_tddft_context}. Moreover, the proof given is based on a constructive procedure, an iterative scheme, enabling a computational approach to find the effective exterior potential.
\end{remark}

\begin{remark}
    We reiterate here that the Coulomb potential \(V\) given by the Poisson equation
    \[
        -\Delta_x V = \rho - 1,
    \]
    is superficial in the sense that it is uniquely determined by the density, which in either cases, injectivity as well as surjectivity, is fixed in the proofs.
\end{remark}

Finally, let us explain in detail the obstruction to working in finite regularity. In the spirit of Hadamard's concept of well-posedness, and goaded by viewing the surjectivity question as an inverse problem, it is natural to demand that the potential-density mapping be not only invertible but also possess some regularity. It is this latter expectation that fails in finite regularity for the inverse of the potential-density mapping. More precisely, we have the following statement.

\begin{proposition}\label{prop_illposedness}
    The density-potential mapping
    \[
        \left(\rho, f|_{t=0}\right)\in C^2_tL^2_{x}\times H^s_{m} \longmapsto \nabla_xV^{\ext}[\rho, f|_{t=0}]\in L^{\infty}_tL^2_{x}
    \]
    is not \(\alpha\)-Hölder continuous for any \(\alpha\in(0,1]\), any \(s, m\in\mathbb{N}\), where \(H^s_m\) is the weighted Sobolev space defined by
    \[
        \norm{f}_{H^s_m} \colonequals \norm{\jap{v}^m f}_{H^s}.
    \]

    Moreover, we can formulate it as an ill-posed inverse problem: there exists an initial datum \(f^0=f^0(v)\) such that \((1, f^0)\in\mathfrak{D}\) and the following holds. For all \(m,s\in \mathbb{N}\), \(\alpha\in (0,1]\), and \(k\in \mathbb{N}\), there are families of solutions \((f_{\varepsilon})_{\varepsilon>0}\) of
    \[
        \begin{cases}
            \partial_t f_{\varepsilon} + v \cdot \nabla_x f_{\varepsilon} + \nabla_x V^{\ext}_{\varepsilon} \cdot \nabla_v f_{\varepsilon} = 0, \\
            \int_{\mathbb{R}^d} f_{\varepsilon} \dd v = 1,
        \end{cases}
    \]
    and times \(t_\varepsilon = O(\varepsilon\abs{\ln \varepsilon})\) such that
    \[
        \lim_{\varepsilon\to 0}\frac{\norm{\nabla_xV^{\ext}_{\varepsilon} - \nabla_xV^{\ext, 0}}_{L^1([0,t_{\varepsilon}];~L^2(\mathbb{T}^d))}}{\norm{\jap{v}^m (f_{\varepsilon}|_{t=0} - f^0)}_{H^s(\mathbb{T}^d\times \mathbb{R}^d)}^\alpha}= +\infty,
    \]
    where \(V^{\ext, 0}\) is the exterior potential associated with the initial datum \(f^0\).
\end{proposition}

We first remark that in this ill-posedness result, one can lose as many derivatives and as much weight in velocity as one wants, and the result still holds.

This ill-posedness result is a consequence of the existence of exponential growing modes for the linearized Vlasov equation around a well-chosen \(f^0\), i.e., satisfying the so-called Penrose instability condition: there exists \(n\in\mathbb{Z}^d\setminus\{0\}\) and \(\omega\in\mathbb{C}^d\) with \(\Im(n\cdot\omega) > 0\) such that
\[
    \int_{\mathbb{R}^d}\frac{n\cdot\nabla_vf^0(v)}{n\cdot(v-\omega)}\dd v = \abs{n}^2.
\]
Typical examples of velocity profiles satisfying these conditions are double-bump equilibria.

Therefore, for a general result on the invertibility of the potential-density mapping, the analytic framework is necessary. Otherwise, we would need to impose a Penrose stability condition \cite{10.1063/1.1706024} to work in finite regularity and avoid the obstruction explained above. This condition, imposed on the velocity profile of the initial datum, \(v\mapsto f^0(x, v)\) for every \(x\), can be understood as an ellipticity condition on the Vlasov operator, i.e., it enables us to invert this very operator---we refer to \cite{MR2863910,MR3306612,MR3592362} for more details.

\subsection{Sketch of proof of \texorpdfstring{\Cref{thm_injectivity_space_analytic}}{Theorem \ref{thm_injectivity_space_analytic}}}

Let us outline the strategy of proof for \Cref{thm_injectivity_space_analytic}. The argument rests on a key elliptic relation obtained from the conservation laws associated with the Vlasov–Poisson system \[
    \begin{cases}\nabla_x\cdot(\rho\nabla_x V^{\ext}) = \nabla_x^2:K[f] - \partial_t^2\rho - \nabla_x\cdot(\rho\nabla_x V), \\
        -\Delta V=\rho-1\end{cases}
\]
where \(J[f]\) and \(K[f]\) denote, respectively, the current and the kinetic energy density associated with \(f\).

Taking the difference for two solutions \(f_1,f_2\) having the same spatial density \(\rho\) but, a priori, different exterior potentials \(V^{\ext}_1,V^{\ext}_2\), and setting \(g\colonequals f_1-f_2\), we obtain
\begin{equation}\label{eq_sketchISA_elliptic}
    \nabla_x\cdot\bigl(\rho\nabla_x(V^{\ext}_1-V^{\ext}_2)\bigr)=\nabla_x^2:K[g].
\end{equation}
From \eqref{eq_sketchISA_elliptic} we reformulate the problem as a uniqueness question for \(g\), which satisfies
\[
    \begin{cases}\partial_t g + v\cdot\nabla_x g + \nabla_x(V+V^{\ext}_1)\cdot\nabla_v g + \nabla_x\bigl(V^{\ext}_1-V^{\ext}_2\bigr)\cdot\nabla_v f_2 = 0, \\ g|_{t=0}=0.\end{cases}
\]
We shall control \(g\) by applying Grönwall's lemma in the analytic norm \(\norm{\cdot}_{M,\lambda,s}\). The main difficulty arises from \eqref{eq_sketchISA_elliptic}: it shows that \(\nabla_x(V^{\ext}_1-V^{\ext}_2)\) has one less order of regularity than \(g\), so the elliptic relation entails no gain of derivatives. This loss is compensated by exploiting the analytic regularity, which permits absorbing the derivative loss in the Grönwall argument by an adequate choice of the analyticity radius \(\lambda\)---we will take it to be time-dependent, precisely, \(t\mapsto\lambda(t)\)  will be an affine function of time.

\subsection{Sketch of proof of \texorpdfstring{\Cref{thm_injectivity_time_analytic}}{Theorem \ref{thm_injectivity_time_analytic}}}

We write a proof from the idea found in \cite{Manfredi_2020} which itself is an adaptation of the Runge--Gross strategy---see \cite{PhysRevA.93.062510,PhysRevLett.52.997}. The proof is based on the Taylor expansion in time of the difference \(V^{\ext}_1-V^{\ext}_2\) at \(t=0\). The coefficients of this expansion can be computed recursively using the Vlasov--Poisson equation and the equality of the densities \(\rho_1=\rho_2\)---the moment bounds ensure that all the terms in the expansion are well-defined.

\subsection{Sketch of proof of \texorpdfstring{\Cref{thm_surjectivity}}{Theorem \ref{thm_surjectivity}}}

For this proof, we once again rely on the elliptic relation between the exterior potential and the spatial density,
\[
    \partial_t^2\rho[f]+\nabla_x\cdot(\rho[f]\nabla_xV^{\ext})=\nabla_x^2:K[f]-\nabla_x\cdot(\rho[f]\nabla_x V),
\]
where \(V^{\ext}\) is normally given. Now, if we reverse our point of view and consider a given spatial density \(\varrho\), a good candidate for an exterior potential that generates this density is the solution of the elliptic equation
\[
    \nabla_x\cdot(\varrho\nabla_x V^{\ext})=\nabla_x^2:K[f]-\partial_t^2\varrho-\nabla_x\cdot(\varrho\nabla_x V).
\]
We are led to consider the coupled Vlasov--Poisson system with unknowns \((f, V^{\ext})\),
\begin{equation}\label{eq_sketchS_VP-surjectivity}
    \begin{cases}
        \partial_tf+v\cdot\nabla_xf+\nabla_xV\cdot\nabla_vf+\nabla_xV^{\ext}\cdot\nabla_vf=0,                        \\
        \nabla_x\cdot(\varrho\nabla_xV^{\ext})=\nabla_x^2:K[f]-\partial_t^2\varrho-\nabla_x\cdot(\varrho\nabla_x V), \\
        -\Delta_xV=\varrho-1,                                                                                        \\
        f|_{t=0}=f^0.
    \end{cases}
\end{equation}
Note that \(V\) can be now considered as given while \(V^{\ext}\) is considered as self-induced. As in the injectivity argument, the potential-density mapping entails a loss of derivatives. We compensate for this loss by working in analytic norms. To construct a solution, we adapt a proof of the Cauchy-Kowalevski theorem similar to the ones found in \cite{MR322321}, \cite{MR1027897}.

\section{Preliminary lemmas}

We first state some useful lemmas---from \cite{MR3489904} for instance.
\begin{lemma}
    For all \(k, \ell, \eta\in\mathbb{R}^d\) and all \(s, \delta > 0\), the following inequalities hold,
    \begin{align}
        \inp{k}{\eta}\leqslant\inp{k-\ell}{\eta}+\jap{\ell}, \label{eq_jap_sublinear}    \\
        \inp{k}{\eta}\lesssim\inp{k-\ell}{\eta}\jap{\ell}, \label{eq_jap_multiplicative} \\
        \jap{k}^{s}\lesssim_{s}\delta^{-s}e^{\delta\jap{k}}. \label{eq_jap_exp}
    \end{align}
\end{lemma}

\begin{lemma}
    Let \(g_k(\eta), h_k(\eta)\in\ell^2_k L^2_{\eta}\) and \(\jap{k}^{\sigma} V_k \in \ell^2_k\) for \(\sigma > d/2\). Then,
    \begin{equation}\label{eq_young1}
        \abs*{\sum_{k, \ell} \int_{\eta} g_k(\eta) V_{\ell} h_{k-\ell}(\eta) \dd\eta}
        \lesssim_{d, \sigma} \norm{g}_{\ell^2_k L^2_{\eta}} \norm{h}_{\ell^2_k L^2_{\eta}} \norm{\jap{\cdot}^{\sigma} V}_{\ell^2_k}.
    \end{equation}
    Let \(g_k(\eta), \jap{k}^{\sigma} h_k(\eta)\in\ell^2_k L^2_{\eta}\) and \(V_k \in \ell^2_k\) for \(\sigma > d/2\). Then,
    \begin{equation}\label{eq_young2}
        \abs*{\sum_{k, \ell} \int_{\eta} g_k(\eta) V_{\ell} h_{k-\ell}(\eta) \dd\eta}
        \lesssim_{d, \sigma} \norm{g}_{\ell^2_k L^2_{\eta}} \norm{\jap{k}^{\sigma} h}_{\ell^2_k L^2_{\eta}} \norm{V}_{\ell^2_k}.
    \end{equation}
\end{lemma}

\begin{lemma}\label{lem_rho_leq_f}
    Let \(f=f(x, v)\) be a function and \(\rho(x)\colonequals\int_{\mathbb{R}^d}f(x, v)\dd v\) its associated density. Then,
    \[
        \norm{\rho}_{\lambda, s} \lesssim \norm{f}_{M, \lambda, s},
    \]
    for every \(M>d/2\).
\end{lemma}

\begin{proof}
    Since \(\hat{\rho}_k = \hat{f}_k|_{\eta=0}\), we have
    \begin{align*}
        \norm{\rho}_{\lambda, s} & =\sum_{k\in\mathbb{Z}^d} \abs*{\hat{f}_k|_{\eta=0}}^2 \jap{k}^{2s} e^{2\lambda \jap{k}}                                                                                    \\
                                 & \lesssim \sum_{\abs{\alpha}\leqslant M}\sum_{k\in\mathbb{Z}^d} \jap{k}^{2s} e^{2\lambda\jap{k}}\int_{\mathbb{R}^d}\abs{D^\alpha_\eta \hat{f}_k(\eta)}^2 \dd\eta            \\
                                 & \lesssim \sum_{\abs{\alpha}\leqslant M}\sum_{k\in\mathbb{Z}^d}\int_{\mathbb{R}^d}\abs{D^\alpha_\eta\hat{f}_k(\eta)}^2\inp{k}{\eta}^{2s} e^{2\lambda \inp{k}{\eta}} \dd\eta \\
                                 & \lesssim \norm{f}_{M, \lambda, s},
    \end{align*}
    where we used the Sobolev embedding \(H^{d/2+}_{\eta}\hookrightarrow L^{\infty}_{\eta}\) in the first inequality.
\end{proof}

Finally, we prove an a priori analytic regularity estimate of solutions to an essential elliptic equation.
\begin{proposition}\label{prop_elliptic_analytic_regularity}
    Let \(d\geqslant 1\), \(\lambda_0>0\), \(s\geqslant 0\), and \(\rho\in\mathcal{A}^{\lambda_0, s}(\mathbb{T}^d)\) such that \(\rho(x)\geqslant\rho_{\min}>0\) for all \(x\in\mathbb{T}^d\). There exists \(0<\lambda_0'=\lambda_0'(d, \rho)<\lambda_0\) such that for all \(0<\lambda<\lambda_0'\), if \(u\in\mathcal{A}^{\lambda, s+2}(\mathbb{T}^d)\) with zero mean is a solution to the elliptic equation
    \[
        \nabla\cdot(\rho\nabla u)=f,
    \]
    where \(f\in\mathcal{A}^{\lambda, s}(\mathbb{T}^d)\). Then, \[
        \norm{u}_{\lambda, s+2}\lesssim_{d, \rho, \lambda_0'}\norm{f}_{\lambda, s}.
    \]
\end{proposition}

\begin{proof}
    From Sobolev elliptic regularity---see \cite{MR4703940}, we have \(\norm{u}_{H^{s+2}}\lesssim_{d, \rho}\norm{f}_{H^s}\). We deduce the analytic regularity estimate,
    \begin{align*}
        \norm{u}_{\lambda, s+2}=\norm{e^{\lambda\jap{\nabla}}u}_{H^{s+2}} & \lesssim_{d, s, \rho} \norm*{\nabla\cdot\left(\rho\nabla \left(e^{\lambda\jap{\nabla}}u\right)\right)}_{H^{s}}                                            \\
                                                                          & \lesssim_{d, s, \rho} \norm*{e^{\lambda\jap{\nabla}}f}_{H^{s}}+\norm*{\nabla\cdot\left(\left[e^{\lambda\jap{\nabla}}, \rho\right]\nabla u\right)}_{H^{s}} \\
                                                                          & \lesssim_{d, s, \rho} \norm{f}_{\lambda, s}+\norm*{\left[e^{\lambda\jap{\nabla}}, \rho\right]\nabla u}_{H^{s+1}} .
    \end{align*}
    It remains to estimate the commutator. We have
    \begin{align*}
        \abs*{\ultrawidehat{[e^{\lambda\jap{\nabla}}, \rho]\nabla u}_k} & \leqslant \sum_{\ell\in\mathbb{Z}^d}\abs*{e^{\lambda\jap{k}} - e^{\lambda\jap{\ell}}}\abs*{\hat{\rho}_{k-\ell}}\abs*{\ell\hat{u}_{\ell}}                                                                                  \\
                                                                        & \leqslant \sum_{\ell\in\mathbb{Z}^d}\abs*{e^{\lambda(\jap{k}-\jap{k-\ell}-\jap{\ell})} - e^{-\lambda\jap{k-\ell}}}\abs*{e^{\lambda\jap{k-\ell}}\hat{\rho}_{k-\ell}}\abs*{\jap{\ell}^2e^{\lambda\jap{\ell}}\hat{u}_{\ell}} \\
                                                                        & \leqslant \sum_{\ell\in\mathbb{Z}^d}\lambda\abs*{\jap{k}-\jap{\ell}}\abs*{e^{\lambda\jap{k-\ell}}\hat{\rho}_{k-\ell}}\abs*{\jap{\ell}e^{\lambda\jap{\ell}}\hat{u}_{\ell}}                                                 \\
                                                                        & \leqslant \lambda\sum_{\ell\in\mathbb{Z}^d}\abs*{\jap{k-\ell}e^{\lambda\jap{k-\ell}}\hat{\rho}_{k-\ell}}\abs*{\jap{\ell}e^{\lambda\jap{\ell}}\hat{u}_{\ell}}.
    \end{align*}
    Now, taking the \(H^{s+1}\)-norm yields, using Young's inequality, for \(\sigma>d/2\) and by \eqref{eq_jap_exp},
    \begin{align*}
        \norm*{\nabla\cdot\left(\left[e^{\lambda\jap{\nabla}}, \rho\right]\nabla u\right)}_{H^{s}} & \leqslant \lambda\norm*{\sum_{\ell\in\mathbb{Z}^d}\abs*{\jap{k-\ell}^{s+2}e^{\lambda\jap{k-\ell}}\hat{\rho}_{k-\ell}}\abs*{\jap{\ell}^{s+2}e^{\lambda\jap{\ell}}\hat{u}_{\ell}}}_{\ell_k^2} \\
                                                                                                   & \leqslant \lambda\norm*{\jap{k}^{s+2}e^{\lambda\jap{k}}\hat{\rho}_{k}}_{\ell_k^1}\norm*{u}_{\lambda, s+2}                                                                                   \\
                                                                                                   & \lesssim_{\sigma} \lambda\norm*{\rho}_{\lambda, s+\sigma+2}\norm*{u}_{\lambda, s+2}                                                                                                         \\
                                                                                                   & \lesssim_{\sigma} \frac{\lambda}{(\lambda_0-\lambda)^{\sigma+2}}\norm*{\rho}_{\lambda_0, s}\norm*{u}_{\lambda, s+2}
    \end{align*}
    Hence, for \(\lambda>0\) small enough, we absorb this last term and obtain the desired estimate.
\end{proof}

\section{Proof of \texorpdfstring{\Cref{thm_injectivity_space_analytic}}{Theorem \ref{thm_injectivity_space_analytic}}}\label{sec_proofISA}

The key of the proof lies in the elliptic equation \eqref{eq_sketchISA_elliptic} which relates the exterior potential \(V^{\ext}\) to the density \(\rho\). Indeed, this equation allows us to reformulate the injectivity of the potential-density mapping as a uniqueness problem for a Vlasov--Poisson type equation.

We set \(g\colonequals f_1-f_2\). By taking the difference in equation \eqref{eq_vlasov-poisson}, we obtain
\begin{equation}\label{eq_proofISA_VP-diff}
    \begin{cases}
        \begin{aligned}
            \partial_tg+v\cdot\nabla_xg & +\nabla_xV\cdot \nabla_vg                                                                   \\
                                        & +\nabla_xV^{\ext}\cdot\nabla_vg+\nabla_x\left(V^{\ext}_1-V^{\ext}_2\right)\cdot\nabla_vf=0,
        \end{aligned} \\
        g|_{t=0}=0,                                                                                                               \\
        -\Delta_xV(t, x)=\rho(t, x)-1,
    \end{cases}
\end{equation} where we denote \(f\colonequals f_1\) and \(V^{\ext}\colonequals V^{\ext}_2\) since there is no distinction in the role played by solutions \(1\) and \(2\).

We first derive the conservation equation by integrating the equation \eqref{eq_proofISA_VP-diff} in \(v\)---notice the \(3\) terms involving potentials are in divergence form---we get, setting
\(
J(t, x)\colonequals\int_{\mathbb{R}^d}vg(t, x, v)\dd v
\)
the current of \(g\),
\[
    \partial_t(\rho_1-\rho_2)+\nabla_x\cdot J=0,\qquad\text{i.e.,}~\nabla_x\cdot J=0,
\]
by virtue of the hypothesis \(\rho_1=\rho_2\). Furthermore, by taking the time derivative of the current, we obtain
\begin{align*}
    \partial_tJ & =\begin{aligned}[t]
                        & -\int_{\mathbb{R}^d}v(v\cdot\nabla_xg)\dd v-\int_{\mathbb{R}^d}v\nabla_xV\cdot\nabla_vg\dd v                                               \\
                        & -\int_{\mathbb{R}^d}v\nabla_x\left(V^{\ext}_1-V^{\ext}_2\right)\cdot\nabla_vf\dd v-\int_{\mathbb{R}^d}v\nabla_xV^{\ext}\cdot\nabla_vg\dd v
                   \end{aligned} \\
                & =-\nabla_x\cdot K(t, x)+\rho(t, x)\nabla_x\left(V^{\ext}_1-V^{\ext}_2\right)(t, x),
\end{align*}
with
\(
K(t, x)\colonequals\int_{\mathbb{R}^d}v\otimes vg(t, x, v)\dd v
\)
the kinetic energy density of \(g\). The last equality was obtained by integration by parts and using the fact that the spatial density of \(g\) is zero. Recall that \(\nabla_x\cdot J(t, x)=0\) according to the conservation equation. We deduce that the difference of the exterior potentials satisfies the elliptic equation
\begin{equation}\label{eq_proofISA_elliptic-diff}
    \nabla_x\cdot(\rho(t, x)\nabla_x\left(V^{\ext}_1-V^{\ext}_2\right)(t, x))=\nabla_x^2:K(t, x).
\end{equation}
Thus, the injectivity of the potential-density mapping is equivalent to the uniqueness of the solution to the following system
\begin{equation}\label{eq_proofISA_uniqueness}
    \begin{cases}
        \partial_tg+v\cdot\nabla_xg+\left(\nabla_xV+\nabla_xV^{\ext}\right)\cdot \nabla_vg+\nabla_x\left(V^{\ext}_1-V^{\ext}_2\right)\cdot\nabla_vf=0, \\
        g|_{t=0}=0,                                                                                                                                    \\
        \nabla_x\cdot(\rho(t, x)\nabla_x\left(V^{\ext}_1-V^{\ext}_2\right)(t, x))=\nabla_x^2:K(t, x),                                                  \\
        -\Delta_xV(t, x)=\rho(t, x)-1.
    \end{cases}
\end{equation}

To prove this uniqueness, we wish to apply Grönwall's lemma, so we need to estimate the quantity \(\partial_t\norm{g(t)}_{M, \lambda(t), s}^2\) where we consider parameters \(M, \lambda(t), s\geqslant 0\) arbitrary. We begin by rewriting equation \eqref{eq_proofISA_uniqueness} in Fourier,
\begin{align*}
    \partial_t & \hat{g}_k(t, \eta) -k\cdot\nabla_{\eta}\hat{g}_k(t, \eta)-(2\pi)^{-d}(k\hat{V}_k(t))\ast_k(\eta\hat{g}_k(t, \eta))                                                                \\
               & -(2\pi)^{-d}k(\hat{V}^{\ext}_k(t))\ast_k(\eta\hat{g}_k(t, \eta))-(2\pi)^{-d}\left[k\ultrawidehat{\left(V^{\ext}_1-V^{\ext}_2\right)}_k(t)\right]\ast_k(\eta\hat{f}_k(t, \eta))=0.
\end{align*}
Since the analyticity radius \(\lambda(t)\) will not be modified during the estimates, we will from now on denote \(A_k^{(s)}(\eta)\colonequals A^{(\lambda(t), s)}_k(\eta)\) to lighten the notation. We compute
\begin{align*}
     & \frac{1}{2} \partial_t\norm{g(t)}_{M, \lambda(t), s}^2=\sum_{\abs{\alpha}\leqslant M}\frac{1}{2}\partial_t\norm{A^{(\lambda(t), s)}D_{\eta}^{\alpha}\hat{g}}_{\ell^2_kL^2_{\eta}}^2                                                                                                                                                                                                           \\
     & =\sum_{\abs{\alpha}\leqslant M}\frac{1}{2}\partial_t\sum_{k\in\mathbb{Z}^d}\int_{\mathbb{R}^d}\abs*{A^{(\lambda(t),s)}_k(\eta)D_{\eta}^{\alpha}\hat{g}_k(t, \eta)}^2\dd\eta                                                                                                                                                                                                                   \\
     & =\begin{aligned}[t]
             & \dot{\lambda}(t)\norm{g(t)}_{M, \lambda(t), s+1/2}^2                                                                                                                                                                                                                                                                  \\
             & +\sum_{\abs{\alpha}\leqslant M}\sum_{k\in\mathbb{Z}^d}\int_{\mathbb{R}^d}A^{(s)}_k(\eta)D_{\eta}^{\alpha}\hat{g}_k(t, \eta)\cdot A^{(s)}_k(\eta)D_{\eta}^{\alpha}[k\cdot\nabla_{\eta}\hat{g}_k(t, \eta)]\dd\eta                                                                                                       \\
             & +\sum_{\abs{\alpha}\leqslant M}\sum_{k\in\mathbb{Z}^d}\int_{\mathbb{R}^d}A^{(s)}_k(\eta)D_{\eta}^{\alpha}\hat{g}_k(t, \eta)\cdot A^{(s)}_k(\eta)D_{\eta}^{\alpha}[(k\hat{V}_k(t))\ast_k(\eta\hat{g}_k(t, \eta))]\frac{\dd\eta}{(2\pi)^d}                                                                              \\
             & +\sum_{\abs{\alpha}\leqslant M}\sum_{k\in\mathbb{Z}^d}\int_{\mathbb{R}^d}A^{(s)}_k(\eta)D_{\eta}^{\alpha}\hat{g}_k(t, \eta)\cdot A^{(s)}_k(\eta)D_{\eta}^{\alpha}[(k(\hat{V}^{\ext})_k(t))\ast_k(\eta\hat{g}_k(t, \eta))]\frac{\dd\eta}{(2\pi)^d}                                                                     \\
             & \begin{aligned}[t]
               + \sum_{\abs{\alpha}\leqslant M}\sum_{k\in\mathbb{Z}^d}\int_{\mathbb{R}^d}  A^{(s)}_k(\eta)D_{\eta}^{\alpha} & \hat{g}_k(t, \eta)                                                                                                                                                \\
                                                                                                                            & \cdot A^{(s)}_k(\eta)\left[k\ultrawidehat{\left(V^{\ext}_1-V^{\ext}_2\right)}_k(t)\right]\ast_k D_{\eta}^{\alpha}(\eta\hat{f}_k(t, \eta))\frac{\dd\eta}{(2\pi)^d}
           \end{aligned}
        \end{aligned} \\
     & \equalscolon\dot{\lambda}(t)\norm{g(t)}_{M, \lambda(t), s+1/2}^2+\I+\II+\II^{\ext}+\III.
\end{align*}
The first term allows compensating the losses of regularity---one of the motivations for our choice to work with analytic norms. It remains to estimate the other four terms. The goal in the estimates is thus to distribute regularity appropriately among the different terms so as not to exceed the loss that can be compensated by this first term.

\begin{claim}\label{claim_proofISA_estimate_I}
    \begin{equation}
        \abs{\I}\lesssim s\norm{g(t)}_{M, \lambda(t),s}^2+\lambda(t)\norm{g(t)}_{M, \lambda(t),s+1/2}^2.
    \end{equation}
\end{claim}

\begin{proof}
    By integrating by parts, we obtain
    \begin{align*}
        \I & =\sum_{\abs{\alpha}\leqslant M}\sum_{k\in\mathbb{Z}^d}\int_{\mathbb{R}^d}A^{(s)}_k(\eta)D_{\eta}^{\alpha}\hat{g}_k(t, \eta)\cdot A^{(s)}_k(\eta)k\cdot\nabla_{\eta}D_{\eta}^{\alpha}\hat{g}_k(t, \eta)\dd\eta                                                                                                       \\
           & =\sum_{\abs{\alpha}\leqslant M}\sum_{k\in\mathbb{Z}^d}\int_{\mathbb{R}^d}k\cdot\frac{1}{2}\nabla_{\eta}\abs{D_{\eta}^{\alpha}\hat{g}_k(t, \eta)}^2A^{(s)}_k(\eta)^2\dd\eta                                                                                                                                          \\
           & =-\sum_{\abs{\alpha}\leqslant M}\sum_{k\in\mathbb{Z}^d}\int_{\mathbb{R}^d}\abs{D_{\eta}^{\alpha}\hat{g}_k(t, \eta)}^2k\cdot\nabla_{\eta}A^{(s)}_k(\eta)A^{(s)}_k(\eta)\dd\eta                                                                                                                                       \\
           & =\begin{aligned}[t]
                  -\sum_{\abs{\alpha}\leqslant M}\sum_{k\in\mathbb{Z}^d}\int_{\mathbb{R}^d} & \abs{D_{\eta}^{\alpha}\hat{g}_k(t, \eta)}^2k\cdot\left[s\eta\inp{k}{\eta}^{s-2}+\inp{k}{\eta}^{s}\lambda(t)\eta\inp{k}{\eta}^{-1}\right] \\
                                                                                            & \cdot e^{\lambda(t)\inp{k}{\eta}}A^{(s)}_k(\eta)\dd\eta
              \end{aligned} \\
           & =-\sum_{\abs{\alpha}\leqslant M}\sum_{k\in\mathbb{Z}^d}\int_{\mathbb{R}^d}\abs{D_{\eta}^{\alpha}\hat{g}_k(t, \eta)}^2\frac{k\cdot\eta}{\inp{k}{\eta}^2}\left[s+\lambda(t)\inp{k}{\eta}\right]\inp{k}{\eta}^{2s}e^{2\lambda(t)\inp{k}{\eta}}\dd\eta.
    \end{align*}
    Thus, since \(\abs*{\frac{k\cdot\eta}{\inp{k}{\eta}^2}}\lesssim 1\), \[
        \abs{\I}\lesssim s\norm{g}_{M, \lambda(t),s}^2+\lambda(t)\norm{g}_{M, \lambda(t),s+1/2}^2.
    \]
\end{proof}

\begin{claim}\label{claim_proofISA_estimate_II}
    For every \(\sigma>d/2\), every \(M>d/2\),
    \begin{equation}
        \begin{aligned}
            \abs{\II} & \lesssim_{d, \sigma}\norm{\rho(t)}_{\lambda(t),s+\sigma-1}\left[\norm{g(t)}_{M,\lambda(t),s}^2+\norm{g(t)}_{M,\lambda(t),s+1/2}^2\right]  \\
                      & \lesssim_{d, \sigma}\norm{f(t)}_{M, \lambda(t),s+\sigma-1}\left[\norm{g(t)}_{M,\lambda(t),s}^2+\norm{g(t)}_{M,\lambda(t),s+1/2}^2\right].
        \end{aligned}
    \end{equation}
\end{claim}

\begin{proof}
    We expand the derivative on the last term
    \begin{align*}
         & \sum_{\abs{\alpha}\leqslant M}\sum_{k\in\mathbb{Z}^d}\int_{\mathbb{R}^d}A^{(s)}_k(\eta)D_{\eta}^{\alpha}\hat{g}_k(t, \eta)\cdot A^{(s)}_k(\eta)\sum_{\ell\in\mathbb{Z}^d}\left[\hat{V}_\ell(t)\ell\cdot D_{\eta}^{\alpha}(\eta\hat{g}_{k-\ell}(t, \eta))\right]\dd\eta                                                                                                          \\
         & =\begin{aligned}[t]
                 & \sum_{\abs{\alpha}\leqslant M}\sum_{k\in\mathbb{Z}^d}\int_{\mathbb{R}^d}A^{(s)}_k(\eta)D_{\eta}^{\alpha}\hat{g}_k(t, \eta)\cdot A^{(s)}_k(\eta)\sum_{\ell\in\mathbb{Z}^d}\left[\hat{V}_\ell(t)\ell\cdot\eta D_{\eta}^{\alpha}\hat{g}_{k-\ell}(t, \eta)\right]\dd\eta \\
                 & + \sum_{\abs{\alpha}\leqslant M}\sum_{k\in\mathbb{Z}^d}\int_{\mathbb{R}^d}A^{(s)}_k(\eta)D_{\eta}^{\alpha}\hat{g}_k(t, \eta)\cdot A^{(s)}_k(\eta)\sum_{\ell\in\mathbb{Z}^d}\sum_{\substack{j\leqslant\alpha                                                          \\ \abs{j}=1}}\left[\hat{V}_\ell(t)\ell_jD_{\eta}^{\alpha-j}\hat{g}_{k-\ell}(t, \eta)\right]\dd\eta
            \end{aligned} \\
         & \equalscolon \II_a+\II_b.
    \end{align*}
    We begin by treating term \(\II_a\). Using \eqref{eq_jap_sublinear}, \eqref{eq_jap_multiplicative}, we notice that \(A_k^{(s)}\lesssim B_{\ell}^{(s)}A_{k-\ell}^{(s)}\). Also, by elliptic regularity, \(\abs{\hat{V}_{\ell}}\leqslant\frac{\abs{\hat{\rho}_{\ell}}}{\abs{\ell}^2}\lesssim\frac{\abs{\hat{\rho}_{\ell}}}{\jap{\ell}^2}\) for all \(\ell\in\mathbb{Z}^d_*\). Then, by Young's inequality \eqref{eq_young1},
    \begin{align*}
        \abs{\II_a} & \lesssim\begin{aligned}[t]
                                  \sum_{\abs{\alpha}\leqslant M}\sum_{k\in\mathbb{Z}^d}\sum_{\ell\in\mathbb{Z}^d_*}\int_{\mathbb{R}^d} & \abs{\eta}^{1/2}\abs{A^{(s)}_k(\eta)D_{\eta}^{\alpha}\hat{g}_k(t, \eta)}                                                               \\
                                                                                                                                       & \cdot\abs{\ell B_{\ell}^{(s)}\hat{V}_{\ell}(t)}\abs{\eta}^{1/2}\abs{A_{k-\ell}^{(s)}D_{\eta}^{\alpha}\hat{g}_{k-\ell}(t, \eta)}\dd\eta\end{aligned} \\
                    & \lesssim\begin{aligned}[t]
                                  \sum_{\abs{\alpha}\leqslant M}\sum_{k\in\mathbb{Z}^d}\sum_{\ell\in\mathbb{Z}^d_*}\int_{\mathbb{R}^d} & \abs{A^{(s+1/2)}_k(\eta)D_{\eta}^{\alpha}\hat{g}_k(t, \eta)}                                                               \\
                                                                                                                                       & \cdot\abs{B_{\ell}^{(s-1)}\hat{\rho}_{\ell}(t)}\abs{A_{k-\ell}^{(s+1/2)}D_{\eta}^{\alpha}\hat{g}_{k-\ell}(t, \eta)}\dd\eta\end{aligned}                 \\
                    & \lesssim_{d, \sigma}\sum_{\abs{\alpha}\leqslant M}\norm{A^{(\lambda(t),s+1/2)}_k(\eta)D_{\eta}^{\alpha}\hat{g}_k(t, \eta)}_{\ell^2_kL^2_{\eta}}^2\norm{\jap{\ell}^{\sigma}B_{k}^{(\lambda(t), s-1)}\hat{\rho}_{k}(t)}_{\ell^2_k}                                                                                                                            \\
                    & \lesssim_{d, \sigma}\norm{\rho}_{\lambda(t),s+\sigma-1}\norm{g}_{M,\lambda(t),s+1/2}^2.
    \end{align*}
    Similarly for term \(\II_b\), using Cauchy-Schwarz inequality, we obtain
    \begin{align*}
        \abs{\II_b} & \lesssim_{d, \sigma}\sum_{\abs{\alpha}\leqslant M}\sum_{j\leqslant\alpha:\abs{j}=1}\norm{v^{\alpha}g}_{\lambda(t), s}\norm{\rho}_{\lambda(t),s+\sigma-1}\norm{v^{\alpha-j}g}_{\lambda(t), s}             \\
                    & \lesssim_{d, \sigma}\norm{\rho}_{\lambda(t),s+\sigma-1}\norm{g}_{M,\lambda(t),s}\left[\sum_{\abs{\alpha}\leqslant M}\sum_{j\leqslant\alpha:\abs{j}=1}\norm{v^{\alpha-j}g}_{\lambda(t), s}^2\right]^{1/2} \\
                    & \lesssim_{d, \sigma}\norm{\rho}_{\lambda(t),s+\sigma-1}\norm{g}_{M,\lambda(t),s}^2.
    \end{align*}
\end{proof}

\begin{claim}
    For every \(\sigma>d/2\),
    \begin{equation}
        \abs{\II^{\ext}}\lesssim_{d, \sigma}\norm{V^{\ext}(t)}_{\lambda(t),s+\sigma+1}\left[\norm{g(t)}_{M,\lambda(t),s}^2+\norm{g(t)}_{M,\lambda(t),s+1/2}^2\right].
    \end{equation}
\end{claim}

\begin{proof}
    The proof is identical to the previous one---\Cref{claim_proofISA_estimate_II}.
\end{proof}

Since only the derivative of the exterior potentials appear in the equation, up to a time dependent constant, we may now assume that, for all \(t\in [0, T]\),
\[
    \int_{\mathbb{T}^d}V^{\ext}_1(t, x)-V^{\ext}_2(t, x)\dd x=0.
\]

\begin{claim}
    For every \(\sigma>d/2\), every \(M>d/2+2\), and for \(0<\lambda(t)<\lambda_0'\) for some \(0<\lambda_0'<\lambda_0\) small enough,
    \begin{equation}
        \begin{aligned}
            \abs{\III} & \lesssim_{d, \sigma, \rho, \lambda_0'}\norm{f(t)}_{M, \lambda(t),s+\sigma+3/2}\norm{g(t)}_{M,\lambda(t),s+1/2}^2
        \end{aligned}
    \end{equation}
\end{claim}

\begin{proof}
    We begin by expanding the derivative \(D_{\eta}^{\alpha}\) and consider the term containing the highest order derivative \(\eta D_{\eta}^{\alpha}f_{k-\ell}(\eta)\); the other lower order terms are handled similarly. We must distribute regularity on the three factors. Using \eqref{eq_jap_multiplicative}, we have \(\jap{\ell}^{1/2}\lesssim\jap{k}^{1/2}\jap{k-\ell}^{1/2}\), and hence
    \[
        \abs*{A^{(s)}_k(\eta)\ell\cdot\eta A^{(s)}_k(\eta)}\lesssim\abs*{\inp{k}{\eta}^{1/2}A^{(s)}_k(\eta)}\abs{\jap{\ell}^{1/2}B^{(s)}_k}\abs*{\inp{k-\ell}{\eta}^{1/2}A^{(s)}_{k-\ell}(\eta)}.
    \]
    Thus, by Young's inequality \eqref{eq_young1},
    \begin{align*}
         & \abs*{\sum_{\abs{\alpha}\leqslant M}\sum_{k,\ell\in\mathbb{Z}^d}\int_{\mathbb{R}^d}A^{(s)}_k(\eta)D_{\eta}^{\alpha}\hat{g}_k(t, \eta)\cdot A^{(s)}_k(\eta)\ultrawidehat{\left(V^{\ext}_1-V^{\ext}_2\right)}_{\ell}(t)\ell\cdot\eta D_{\eta}^{\alpha}\hat{f}_{k-\ell}(t, \eta)\dd\eta} \\
         & \lesssim\sum_{\abs{\alpha}\leqslant M}\norm{A^{(s+1/2)}D_{\eta}^{\alpha}\hat{g}_k(t)}_{\ell^2_kL^2_{\eta}}\norm{B^{(s+1/2)}\ultrawidehat{\left(V^{\ext}_1-V^{\ext}_2\right)}(t)}_{\ell^2_k}\norm{A^{(s+\sigma+3/2)}D_{\eta}^{\alpha}\hat{f}(t)}_{\ell^2_kL^2_{\eta}}                  \\ 
         & \lesssim_{d, \sigma}\norm{f(t)}_{M, \lambda(t),s+\sigma+3/2}\norm{g(t)}_{M,\lambda(t),s+1/2}\norm{(V^{\ext}_1-V^{\ext}_2)(t)}_{\lambda(t), s+1/2}.
    \end{align*}

    Moreover, by elliptic regularity---\Cref{prop_elliptic_analytic_regularity}---of equation \eqref{eq_proofISA_elliptic-diff}, we have
    \[
        \norm{(V^{\ext}_1-V^{\ext}_2)(t)}_{\lambda(t), s}\lesssim_{d, \rho, \lambda_0'}\norm{K(t)}_{\lambda(t), s}\lesssim_{d, \rho}\norm{g(t)}_{M, \lambda(t), s},
    \]
    for every \(M>d/2+2\) proceeding as in \Cref{lem_rho_leq_f}. Hence,
    \[
        \abs{\III}\lesssim_{d, \sigma, \rho, \lambda_0'}\norm{f(t)}_{M, \lambda(t),s+\sigma+3/2}\norm{g(t)}_{M,\lambda(t),s+1/2}^2.
    \]
\end{proof}

Finally, gathering the previous estimates,
\begin{align*}
     & \partial_t\norm{g(t)}_{M, \lambda(t), s}^2\lesssim_{\substack{d, s, \sigma,                                                                                      \\ \rho, \lambda_0'}}\left(1+\norm{f(t)}_{M, \lambda(t),s+\sigma+1}+\norm{V^{\ext}(t)}_{\lambda(t),s+\sigma-1}\right)\norm{g(t)}_{M,\lambda(t),s}^2 \\
     & +\left(\dot{\lambda}(t)+\lambda(t)+\norm{f(t)}_{M, \lambda(t),s+\sigma+3/2}+\norm{V^{\ext}(t)}_{\lambda(t),s+\sigma+1}\right)\norm{g(t)}_{M,\lambda(t),s+1/2}^2.
\end{align*}

Thus, by choosing \(\lambda(t)=-(\lambda_0'+C)t+\frac{\lambda_0'}{2}\) an affine decreasing function---where we take \(\delta>0\) small enough in \eqref{eq_jap_exp} such that \(\norm{g(t)}_{M, \lambda(t), s+1/2}\) is well-defined---the second term is nonpositive. We can now apply Grönwall's lemma using \(g|_{t=0}=0\) to obtain that \(g=0\) on \([0, \tau]\) where \(\tau\) is the time of vanishing of \(\lambda\). We iterate this process to cover the whole \([0, T]\)---applying Grönwall's lemma with translated initial times from \(\tau\) at each step. This concludes the proof of the theorem.

\section{Proof of \texorpdfstring{\Cref{thm_injectivity_time_analytic}}{Theorem \ref{thm_injectivity_time_analytic}}}

We adapt the proof idea from in \cite{Manfredi_2020, PhysRevA.93.062510}. Recall that \(g\colonequals f_1-f_2\) satisfies
\begin{equation}\label{eq_proofITA_VP-diff}
    \begin{cases}
        \partial_tg+v\cdot\nabla_xg+\nabla_xV\cdot \nabla_vg+\nabla_x\left(V^{\ext}_1-V^{\ext}_2\right)\cdot\nabla_vf+\nabla_xV^{\ext}\cdot\nabla_vg=0, \\
        g|_{t=0}=0,
    \end{cases}
\end{equation} where we denote \(f\colonequals f_1\) and \(V^{\ext}\colonequals V^{\ext}_2\) since there is no distinction in the role played by solutions \(1\) and \(2\). We consider the following quantities
\[
    \begin{cases}J(t, x)\colonequals\int_{\mathbb{R}^d}vg(t, x, v)\dd v, \\ K^k(t, x)\colonequals\int_{\mathbb{R}^d}v^{\otimes k}\otimes vg(t, x, v)\dd v.\end{cases}
\]
We proceed by induction on \(k\in\mathbb{N}\), proving the following property for fixed \(k\in\mathbb{N}\):
\[
    \forall\ell\leqslant k,~\begin{cases}\forall x\in\mathbb{T}^d,~\nabla_x\partial_t^{\ell} (V^{\ext}_1 - V^{\ext}_2)(0, x)=0, \\ \partial_t^{\ell+1}J(t, x)=\rho\nabla_x\partial_t^{\ell}(V^{\ext}_1 - V^{\ext}_2)(t, x) \pm\nabla_x^{\ell+1}:K^{\ell+1}(t, x)+R_{\ell}(t, x), \\ \forall j\leqslant k+1-\ell,~\partial_t^{j}R_{\ell}|_{t=0}=0,\end{cases}
\]
where \(\nabla_x^{\ell}\colonequals(\nabla_x\cdot)^{\ell}\) is the contraction of order \(\ell\). We already know that \(R_0=0\).

Let \(\varphi\in C^{\infty}(\mathbb{T}^d)\). Using equation \eqref{eq_proofITA_VP-diff}, we have
\begin{align*}
    0 & =\partial_t^2\int_{\mathbb{T}^d}(\rho_1-\rho_2)(t, x)\varphi(x)\dd x                                                                                                     \\
      & =-\int_{\mathbb{T}^d}\nabla_x\cdot\partial_tJ(t, x)\varphi(x)\dd x                                                                                                       \\
      & =\int_{\mathbb{T}^d}\partial_tJ(t, x)\cdot\nabla_x\varphi(x)\dd x                                                                                                        \\
      & =\int_{\mathbb{T}^d}\rho(t, x)\nabla_x(V^{\ext}_1-V^{\ext}_2)(t, x)\cdot\nabla_x\varphi(x)\dd x - \int_{\mathbb{T}^d}\nabla_x\cdot K^1(t, x)\cdot\nabla_x\varphi(x)\dd x
\end{align*}
Hence, at \(t=0\) and for \(\varphi(x)=(V^{\ext}_1-V^{\ext}_2)(0, x)\), we have
\[
    0=\int_{\mathbb{T}^d}\rho(0, x)\abs*{\nabla_x(V^{\ext}_1-V^{\ext}_2)}^2(0, x)\dd x.
\]
Therefore, since \(\rho\geqslant 0\) and \(\{\rho(0, x) = 0\}\) is negligible, \(\nabla_x\partial_t^0(V^{\ext}_1 - V^{\ext}_2)=0\), which initializes the induction. Let \(k\in\mathbb{N}^*\) be fixed and assume the property is true for \(k-1\). Then, by the induction hypothesis,
\[
    \partial_t^{k+1}J=\rho\nabla_x\partial_t^{k}(V^{\ext}_1 - V^{\ext}_2) +\nabla_x^{k}:\partial_tK^{k}+\partial_t\rho\nabla_x\partial_t^{k-1}(V^{\ext}_1-V^{\ext}_2)+\partial_tR_{k-1}
\]
and, again by \eqref{eq_proofITA_VP-diff}, we obtain \[
    \partial_tK^{k} =\begin{aligned}[t]
         & -\nabla_x\cdot K^{k+1}-\int_{\mathbb{T}^d}v^{\otimes k}\otimes v(\nabla_x(V^{\ext}_1-V^{\ext}_2)\cdot\nabla_vf)\dd v                                                                          \\
         & +\sum_{j=1}^{k+1}\int_{\mathbb{R}^d}v\otimes\cdots\otimes v\otimes \underbrace{(\nabla_x V+\nabla_x V^{\ext})}_{j\text{\textsuperscript{th} position}}\otimes v\otimes \cdots\otimes vg\dd v.
    \end{aligned}
\]
We then set \[
    \pm R_k\colonequals\begin{aligned}[t]
         & \sum_{j=1}^{k+1}\int_{\mathbb{R}^d}v\otimes\cdots\otimes v\otimes \overbrace{(\nabla_x V+\nabla_x V^{\ext})}^{j\text{\textsuperscript{th} position}}\otimes v\otimes\cdots\otimes vg\dd v \\
         & -\int_{\mathbb{T}^d}v^{\otimes k}\otimes v(\nabla_x(V^{\ext}_1-V^{\ext}_2)\cdot\nabla_vf)\dd v                                                                                            \\
         & +\partial_t\rho\nabla_x\partial_t^{k-1}(V^{\ext}_1-V^{\ext}_2)+\partial_tR_{k-1},
    \end{aligned}
\]
which satisfies \(R_k|_{t=0}=0\) and
\[
    \partial_t^{k+1}J(t, x)=\rho\nabla_x\partial_t^{k}(V^{\ext}_1 - V^{\ext}_2)(t, x) \pm\nabla_x^{k+1}:K^{k+1}(t, x)+R_k(t, x).
\]
Thus, proceeding as in the initialization, for \(\varphi\in C^{\infty}(\mathbb{T}^d)\), we have
\begin{align*}
    0 & =\partial_t^{k+2}\int_{\mathbb{T}^d}(\rho_1-\rho_2)(t, x)\varphi(x)\dd x                                                    \\
      & =-\int_{\mathbb{T}^d}\nabla_x\cdot\partial_t^{k+1}J(t, x)\varphi(x)\dd x                                                    \\
      & =\begin{aligned}[t]
              & \int_{\mathbb{T}^d}\rho(t, x)\nabla_x\partial_t^k(V^{\ext}_1-V^{\ext}_2)(t, x)\cdot\nabla_x\varphi(x)\dd x \\
              & \pm \int_{\mathbb{T}^d}\nabla_x^{k+1}: K^{k+1}(t, x)\cdot\nabla_x\varphi(x)\dd x                           \\
              & + \int_{\mathbb{T}^d}R_k(t, x)\cdot\nabla_x\varphi(x)\dd x.
         \end{aligned}
\end{align*}
Hence, at \(t=0\) and for \(\varphi(x)=\partial_t^k(V^{\ext}_1-V^{\ext}_2)(0, x)\), we have
\[
    0=\int_{\mathbb{T}^d}\rho(0, x)\abs*{\nabla_x\partial_t^k(V^{\ext}_1-V^{\ext}_2)}^2(0, x)\dd x.
\]
Therefore, since \(\rho\geqslant 0\) and \(\{\rho(0, x) = 0\}\) is negligible, \(\nabla_x\partial_t^k(V^{\ext}_1 - V^{\ext}_2)=0\). It remains to verify that \(\partial_t^{k+1-\ell}R_{\ell}|_{t=0}=0\) successively for \(\ell=1,\ldots,k\). We have \[
    \pm\partial_t^{k+1-\ell}R_{\ell}\colonequals\begin{aligned}[t]
         & \sum_{j=1}^{\ell+1}\int_{\mathbb{R}^d}v\otimes\cdots\otimes v\otimes \overbrace{(\nabla_x V+\nabla_x V^{\ext})}^{j\text{\textsuperscript{th} position}}\otimes v\otimes\cdots\otimes v\partial_t^{k+1-\ell}g\dd v \\
         & -\int_{\mathbb{T}^d}v^{\otimes \ell}\otimes v\partial_t^{k+1-\ell}(\nabla_x(V^{\ext}_1-V^{\ext}_2)\cdot\nabla_vf)\dd v                                                                                            \\
         & +\partial_t^{k+1-\ell}(\partial_t\rho\nabla_x\partial_t^{\ell-1}(V^{\ext}_1-V^{\ext}_2))+\partial_t^{k+1-(\ell-1)}R_{\ell-1}.
    \end{aligned}
\]
We analyze the terms from last to first:
\begin{itemize}
    \item \(\partial_t^{k+1-(\ell-1)}R_{\ell-1}|_{t=0}=0\) by the step \(\ell - 1\),

    \item \(\partial_t^{k+1-\ell}(\partial_t\rho\nabla_x\partial_t^{\ell-1}(V^{\ext}_1-V^{\ext}_2))|_{t=0}=0\) by the induction hypothesis, since \(\nabla_x\partial_t^k(V^{\ext}_1-V^{\ext}_2)\) is the highest order derivative,

    \item likewise, \(\int_{\mathbb{T}^d}v^{\otimes \ell}\otimes v\partial_t^{k+1-\ell}(\nabla_x(V^{\ext}_1-V^{\ext}_2)\cdot\nabla_vf)\dd v|_{t=0}=0\),

    \item finally, for the first term, by replacing time derivatives of \(g\) using equation \eqref{eq_proofITA_VP-diff}, we obtain only product terms involving \(g\) or \(\nabla_x\partial_t^{\ell}(V^{\ext}_1-V^{\ext}_2)\) whose time derivative order is at most \(k\). Thus, by the induction hypothesis and \(g|_{t=0}=0\), the first term is also zero at \(t=0\).
\end{itemize}
Finally, \(\partial_t^{k+1-\ell}R_{\ell}|_{t=0}=0\), which yields the induction step: this concludes the proof.

\section{Proof of \texorpdfstring{\Cref{thm_surjectivity}}{Theorem \ref{thm_surjectivity}} and \texorpdfstring{\Cref{prop_illposedness}}{Proposition \ref{prop_illposedness}}}

\subsection{Proof of \texorpdfstring{\Cref{thm_surjectivity}}{Theorem \ref{thm_surjectivity}}}

We proceed by an iteration scheme to solve \eqref{eq_sketchS_VP-surjectivity}, defining a sequence of approximations \((f^{n}, V^{\ext, n})_{n\in\mathbb{N}}\) with \(f^{0}\) the initial datum as follows
\begin{equation}\label{eq_proofS_VP-iterated}
    \begin{cases}
        \partial_tf^{n+1} + v\cdot\nabla_xf^{n+1} = - \nabla_xV\cdot\nabla_vf^{n}- \nabla_xV^{\ext,n}\cdot\nabla_vf^{n}, \\
        \nabla_x\cdot(\varrho\nabla_xV^{\ext,n})=\nabla_x^2:K[f^n]-\partial_t^2\varrho-\nabla_x\cdot(\varrho\nabla_xV),  \\
        -\Delta_xV=\varrho-1,                                                                                            \\
        f^{n+1}|_{t=0}=f^{0}.
    \end{cases}
\end{equation}
At each step \(V^{\ext, n}\) existence is given through the elliptic equation thanks to the coercivity of \(\varrho\). We also consider the difference \(g^{n+1}=f^{n+1}-f^{n}\) satisfying
\begin{equation}\label{eq_proofS_VP-iterated-diff}
    \begin{cases}
        \partial_tg^{n+1} + v\cdot\nabla_xg^{n+1}=\begin{aligned}[t]
                                                       & -\nabla_xV\cdot\nabla_vg^{n} - \nabla_xV^{\ext,n}\cdot\nabla_vg^{n} \\
                                                       & -\nabla_x(V^{\ext,n}-V^{\ext,n-1})\cdot\nabla_vf^{n-1},
                                                  \end{aligned} \\
        \nabla_x\cdot(\varrho\nabla_x(V^{\ext,n}-V^{\ext,n-1}))=\nabla_x^2:K[g^{n}],                                     \\
        -\Delta_xV=\varrho-1,                                                                                            \\
        g^{n+1}|_{t=0}=0.
    \end{cases}
\end{equation}
Thus, at each step, we construct \(g^n\) then \(f^n\) as
\[
    g^n(t, x, v)=\int_0^t \partial_\tau[g^n(\tau, x - ( t - \tau) v, v)]\dd\tau, \qquad f^n=f^{0}+\sum_{j=1}^{n}g^j.
\]
To prove the convergence of the scheme, we introduce the following norm à la Caflisch \cite{MR1027897}, for \(M>2+d/2\), \(s>1+d/2\), \(\lambda_0'>0\), \(c>\lambda_0'\) and \(\gamma\in(0,1)\),
\[
    \normop{u}\colonequals\sup_{\substack{0\leqslant\lambda<\lambda_0' \\ t<\frac{\lambda_0'-\lambda}{c}}}\left\{\norm{u(t)}_{M, \lambda, s} + (\lambda_0'-\lambda - ct)^{\gamma}\left(\norm{\nabla_xu(t)}_{M, \lambda, s} + \norm{\nabla_vu(t)}_{M, \lambda, s}\right) \right\}.
\]
Choose \(0<\lambda_0'\leqslant\lambda_0\) and \(R>0\) such that \(\normop{g^1}\leqslant R/2\). To do so, we need to estimate \(g^1\), which is done by proceeding similarly as in the induction step for the convergence of the scheme, thus, we only write the details for the induction step.

We start by estimating the free transport semigroup.
\begin{claim}\label{claim_proofS_semigroup}
    For \(g(t, x, v)=f(t, x + tv, v)\), we have
    \[
        \norm{g}_{M, \lambda, s}\leqslant (1+\abs{t})^s\norm{f}_{M, \lambda(1 + \abs{t}), s}.
    \]
\end{claim}

\begin{proof}
    By the change of variables \(\xi=\eta-kt\), we have
    \begin{align*}
        \norm{g}_{M,\lambda,s}^2
         & = \sum_{|\alpha|\leqslant M}\sum_{k\in\mathbb{Z}^d}\int_{\mathbb{R}^d} \abs{D_\eta^\alpha \hat f_k(\eta-kt)}^2 \inp{k}{\eta}^{2s} e^{2\lambda \inp{k}{\eta}}\dd\eta \\
         & = \sum_{|\alpha|\leqslant M}\sum_{k\in\mathbb{Z}^d}\int_{\mathbb{R}^d} \abs{D_\xi^\alpha \hat f_k(\xi)}^2 \inp{k}{\xi+kt}^{2s} e^{2\lambda \inp{k}{\xi+kt}}\dd\xi,
    \end{align*}
    and get the bound by using \(\inp{k}{\xi+kt}\leqslant (1+\abs{t})\inp{k}{\xi}\).
\end{proof}

We now estimate \(\normop{g^{n+1}}\). Let \(0\leqslant\lambda<\lambda_0'\). From \eqref{eq_proofS_VP-iterated-diff}, we have
\begin{align*}
    \norm{g^{n+1}(t)}_{M, \lambda, s} & \leqslant \begin{aligned}[t]
                                                      \int_0^t & \norm{\nabla_xV\cdot\nabla_vg^{n}(\tau, \cdot_x - (t - \tau)\cdot_v, \cdot_v)}_{M, \lambda, s}                                     \\
                                                               & + \norm{\nabla_xV^{\ext,n}\cdot\nabla_vg^{n}(\tau, \cdot_x - (t - \tau)\cdot_v, \cdot_v))}_{M, \lambda, s}                         \\
                                                               & + \norm{\nabla_x(V^{\ext,n}-V^{\ext,n-1})\cdot\nabla_vf^{n-1}(\tau, \cdot_x - (t - \tau)\cdot_v, \cdot_v))}_{M, \lambda, s}\dd\tau
                                                  \end{aligned} \\
                                      & \equalscolon \int_0^t \I + \II + \III\dd\tau.
\end{align*}

By \Cref{claim_proofS_semigroup}, Cauchy-Schwarz inequality and Sobolev injection for \(s>1+d/2\), we get
\begin{align*}
    \abs{\I}  & \lesssim_{d, s, \varrho, T}\norm*{\nabla_vg^{n}(\tau)}_{M, \lambda(1 + t - \tau), s},                                                                                      \\
    \abs{\II} & \lesssim_{d, s, \varrho, T} \begin{aligned}[t]
                                                 & \norm*{V^{\ext,n}(\tau)}_{\lambda(1 + t - \tau), s}\norm*{\nabla_vg^{n}(\tau)}_{M, \lambda(1 + t - \tau), s}    \\
                                                 & + \norm*{\nabla_xV^{\ext,n}(\tau)}_{\lambda(1 + t - \tau), s}\norm*{g^{n}(\tau)}_{M, \lambda(1 + t - \tau), s},
                                            \end{aligned}
\end{align*}
and
\[
    \abs{\III}  \lesssim_{d, s, \varrho, T} \begin{aligned}[t]
         & \norm*{(V^{\ext, n}-V^{\ext,n-1})(\tau)}_{\lambda(1 + t - \tau), s}\norm*{\nabla_vf^{n-1}(\tau)}_{M, \lambda(1 + t - \tau), s}    \\
         & + \norm*{\nabla_x(V^{\ext, n}-V^{\ext,n-1})(\tau)}_{\lambda(1 + t - \tau), s}\norm*{f^{n-1}(\tau)}_{M, \lambda(1 + t - \tau), s}.
    \end{aligned}
\]
Moreover, by taking \(\lambda_0'>0\) smaller if necessary---so that \Cref{prop_elliptic_analytic_regularity} applies to \eqref{eq_proofS_VP-iterated} and \eqref{eq_proofS_VP-iterated-diff}, we have
\[
    \abs{\II} \lesssim_{d, s, \varrho, T} \begin{aligned}[t]
         & \bigl(\begin{aligned}[t]
                      & \norm*{f^{n}(\tau)}_{M, \lambda(1 + t - \tau), s} + \norm{\partial_t^2\varrho(\tau)}_{M, \lambda(1 + t - \tau), s}  \\
                      & + \norm{\varrho(\tau)}_{M, \lambda(1 + t - \tau), s}\bigr)\norm*{\nabla_vg^{n}(\tau)}_{M, \lambda(1 + t - \tau), s}
                 \end{aligned}                  \\
         & + \bigl(\begin{aligned}[t]
                        & \norm*{\nabla_xf^{n}(\tau)}_{M, \lambda(1 + t - \tau), s} + \norm{\nabla_x\partial_t^2\varrho(\tau)}_{M, \lambda(1 + t - \tau), s} \\
                        & + \norm{\nabla_x\varrho(\tau)}_{M, \lambda(1 + t - \tau), s}\bigr)\norm*{g^{n}(\tau)}_{M, \lambda(1 + t - \tau), s},
                   \end{aligned}
    \end{aligned}
\]
and
\[
    \abs{\III}  \lesssim_{d, s, \varrho, T} \begin{aligned}[t]
         & \norm*{g^{n}(\tau)}_{M, \lambda(1 + t - \tau), s}\norm*{\nabla_vf^{n-1}(\tau)}_{M, \lambda(1 + t - \tau), s}   & \\
         & +\norm*{\nabla_xg^{n}(\tau)}_{M, \lambda(1 + t - \tau), s}\norm*{f^{n-1}(\tau)}_{M, \lambda(1 + t - \tau), s}.
    \end{aligned}
\]
Hence, we obtain
\begin{align*}
    \norm*{g^{n+1}(t)}_{M, \lambda, s} & \lesssim_{d, s, \varrho, T} (\normop{f^{n-1}} + \normop{f^{n}} + 1)\normop{g^{n}}\int_0^t(\lambda_0' - \lambda - \lambda t - (c - \lambda)\tau)^{-\gamma}\dd\tau \\
                                       & \lesssim_{d, s, \varrho, T} (\normop{f^{n-1}} + \normop{f^{n}} + 1)\normop{g^{n}}\frac{(\lambda_0' - \lambda - \lambda t)^{1-\gamma}}{(1 - \gamma)(c - \lambda)} \\
                                       & \lesssim_{d, s, \varrho, \lambda_0', \gamma, T} \frac{1}{c - \lambda}(\normop{f^{n-1}} + \normop{f^{n}} + 1)\normop{g^{n}}.
\end{align*}
Similarly, we estimate \(\norm*{\nabla_xg^{n+1}(t)}_{M, \lambda, s}\) by tracking the changes made to the analyticity radius. First, we use the fact that \([\nabla_x, v\cdot\nabla_x]=0\) to apply \Cref{claim_proofS_semigroup} which replaces \(\lambda\) by \(\lambda(1+t - \tau)\). Then, we absorb the extra derivative in \(x\) with \eqref{eq_jap_exp}: \(\lambda(1+t - \tau)\) becomes \(\frac{\lambda_0'-c\tau}{2} + \frac{\lambda(1 + t - \tau)}{2}\). Thus, we get
\begin{align*}
    \norm*{\nabla_xg^{n+1}(t)}_{M, \lambda, s} & \lesssim_{d, s, \varrho, \gamma, T} (\normop{f^{n-1}} + \normop{f^{n}} + 1)\normop{g^{n}}\int_0^t(\lambda_0' - \lambda - \lambda t - (c - \lambda)\tau)^{-1-\gamma}\dd\tau \\
                                               & \lesssim_{d, s, \varrho, \gamma, T} \frac{(\lambda_0' - \lambda - c t)^{-\gamma}}{c - \lambda}(\normop{f^{n-1}} + \normop{f^{n}} + 1)\normop{g^{n}}.
\end{align*}
It only remains to estimate \(\norm{\nabla_vg^{n+1}(t)}_{M, \lambda, s}\). We first absorb the derivative in \(v\), \(\lambda\) is replaced by \(\frac{\lambda_0' + \lambda + ct - 2c\tau}{2(1 + t - \tau)}\), then, applying \Cref{claim_proofS_semigroup}, it becomes \(\frac{\lambda_0' + \lambda + ct - 2c\tau}{2}\) to yield the following estimate
\begin{align*}
    \norm*{\nabla_vg^{n+1}(t)}_{M, \lambda, s} & \lesssim_{d, s, \varrho, \gamma, T} (\normop{f^{n-1}} + \normop{f^{n}} + 1)\normop{g^{n}}\int_0^t\frac{(\lambda_0' - \lambda - ct)^{-\gamma}(1 + t - \tau)}{\lambda_0' + \lambda + ct - 2\lambda(1 + t - \tau)- 2c\tau}\dd\tau \\
                                               & \lesssim_{d, s, \varrho, \gamma, T} \frac{(\lambda_0' - \lambda - ct)^{-\gamma}}{c - \lambda}(\normop{f^{n-1}} + \normop{f^{n}} + 1)\normop{g^{n}}.
\end{align*}
Finally, by combining the three estimates and induction, \(\normop{f^{n-1}}, \normop{f^{n}}\leqslant R\), we get
\[
    \normop{g^{n+1}}\lesssim_{d, s, \varrho, \lambda_0', \gamma, T} \frac{1}{c - \lambda_0'}(2R + 1)\normop{g^{n}} = C(d, s, \varrho, \lambda_0', \gamma, T, R, c)\normop{g^{n}}.
\]
Therefore, we can choose \(c>\lambda_0'\) large enough so that \(C(d, s, \varrho, \lambda_0', \gamma, T, R, c)\leqslant 1/2\), yielding the contraction \(\normop{g^{n+1}}\leqslant\normop{g^n}/2\leqslant R/2^{n+2}\) and
\[
    \normop{f^{n+1}}\leqslant\sum_{k=1}^{n+1}\normop{g^k}\leqslant R.
\]
This completes the induction. Since \(\normop{\cdot}\) generates a Banach space and \((g^n)_{n\in\mathbb{N}^*}\) is normally convergent in this space, the sequence \((f^{n})_{n\in\mathbb{N}}\) converges to some \(f\) in the norm \(\normop{\cdot}\), solving \eqref{eq_sketchS_VP-surjectivity}, and, we get the existence of \(V^{\ext}\) through the elliptic equation
\[
    \nabla_x\cdot(\varrho\nabla_xV^{\ext})=\nabla_x^2:K[f]-\partial_t^2\varrho-\nabla_x\cdot(\varrho\nabla_xV).
\]
Finally, we check that the spatial density of \(f\), denoted \(\rho[f]\), is indeed \(\varrho\). As before, \(\rho[f]\) solves the equation \(\partial_t^2\rho[f]+\nabla_x\cdot(\rho[f]\nabla_x(V^{\ext} + V))=\nabla_x^2:K[f]\) with initial datum \(\rho[f]|_{t=0}=\varrho|_{t=0}\) and \(\partial_t\rho[f]|_{t=0}=\partial_t\varrho|_{t=0}\), thanks to the conservation equation. Since the linear operator \(\rho\mapsto \nabla_x\cdot(\rho\nabla_x(V^{\ext} + V))\) satisfies
\[
    \norm{\nabla_x\cdot(\rho\nabla_x(V^{\ext} + V))}_{\lambda, s}\lesssim_{d, s, \varrho}\frac{\norm{\rho}_{\lambda', s}}{\lambda' - \lambda}\qquad\text{for }\lambda<\lambda'<\lambda_0,
\]
by uniqueness of analytic solutions to this equation---see \cite[Theorem~A]{MR322321}---we have \(\rho[f]=\varrho\), which concludes the proof.

\subsection{Proof of \texorpdfstring{\Cref{prop_illposedness}}{Proposition \ref{prop_illposedness}}}

We start by using the result from \cite[Theorem~1.2]{MR3509003}: there exists an initial datum \(f^0=f^0(v)\) such that \((1, f^0)\in\mathfrak{D}\) and the following holds. For all \(m,s\in \mathbb{N}\), \(\alpha\in (0,1]\), and \(k\in \mathbb{N}\), there are families of solutions \((f_{\varepsilon})_{\varepsilon>0}\) of
\[
    \begin{cases}
        \partial_t f_{\varepsilon} + v \cdot \nabla_x f_{\varepsilon} + \nabla_x V^{\ext}_{\varepsilon} \cdot \nabla_v f_{\varepsilon} = 0, \\
        \int_{\mathbb{R}^d} f_{\varepsilon} \dd v = 1,
    \end{cases}
\]
times \(t_\varepsilon = O(\varepsilon\abs{\ln \varepsilon})\), and \((x_0,v_0)\in \mathbb{T}^d\times \mathbb{R}^d\), such that
\[
    \lim_{\varepsilon\to 0}\frac{\norm{f_{\varepsilon} - f^0}_{L^2([0,t_{\varepsilon}]\times \Omega_{\varepsilon})}}{\norm{\jap{v}^m (f_{\varepsilon}|_{t=0} - f^0)}_{H^s(\mathbb{T}^d\times \mathbb{R}^d)}^\alpha} = +\infty,
\]
with \(\Omega_{\varepsilon} = B(x_0,\varepsilon^k)\times B(v_0,\varepsilon^k)\).

Now, we estimate \(\norm{f_{\varepsilon}(t) - f^0}_{L^2_{x, v}}\). We have
\begin{align*}
    \frac{1}{2}\partial_t\norm{f_{\varepsilon}(t) - f^0}_{L^2_{x, v}}^2 & = -\int_{\mathbb{T}^d\times \mathbb{R}^d}(f_{\varepsilon}(t) - f^0)(v\cdot\nabla_x(f_{\varepsilon}(t) - f^0) + \nabla_x V^{\ext}_{\varepsilon}(t)\cdot\nabla_v(f^0 + f_{\varepsilon}(t) - f^0))\dd x\dd v                                                                                 \\
                                                                        & = \begin{aligned}[t]
                                                                                 & -\frac{1}{2}\int_{\mathbb{T}^d\times \mathbb{R}^d}\nabla_x\cdot v\abs{f_{\varepsilon}(t) - f^0}^2\dd x\dd v - \int_{\mathbb{T}^d\times \mathbb{R}^d}(f_{\varepsilon}(t) - f^0)\nabla_x V^{\ext}_{\varepsilon}(t)\cdot\nabla_vf^0\dd x\dd v \\
                                                                                 & - \frac{1}{2}\int_{\mathbb{T}^d\times \mathbb{R}^d}\nabla_v\cdot\nabla_x V^{\ext}_{\varepsilon}(t)\abs{f_{\varepsilon}(t) - f^0}^2\dd x\dd v
                                                                            \end{aligned} \\
                                                                        & \leqslant \norm{\nabla_x V^{\ext}_{\varepsilon}(t)}_{L^2_{x}}\norm{\nabla_vf^0}_{L^2_{v}}\norm{f_{\varepsilon}(t) - f^0}_{L^2_{x, v}},
\end{align*}
So, by integrating in time, we get
\begin{align*}
    \norm{f_{\varepsilon}(t) - f^0}_{L^2_{x, v}} & \leqslant \norm{f_{\varepsilon}|_{t=0} - f^0}_{L^2_{x, v}} + \int_0^t\norm{\nabla_x V^{\ext}_{\varepsilon}(\tau)}_{L^2_{x}}\norm{\nabla_vf^0}_{L^2_{v}}\dd\tau \\
                                                 & \leqslant \norm{f_{\varepsilon}|_{t=0} - f^0}_{L^2_{x, v}} + \norm{\nabla_vf^0}_{L^2_{v}}\norm{\nabla_x V^{\ext}_{\varepsilon}}_{L^1_{t} L^2_x}.
\end{align*}
Then, taking the \(L^2\)-norm in time and dividing by \(\norm{\jap{v}^m (f_{\varepsilon}|_{t=0} - f^0)}_{H^s(\mathbb{T}^d\times \mathbb{R}^d)}^\alpha\), we have
\begin{align*}
    \frac{\norm{f_{\varepsilon} - f^0}_{L^2([0,t_{\varepsilon}]\times\Omega_{\varepsilon})}}{\norm{\jap{v}^m (f_{\varepsilon}|_{t=0} - f^0)}_{H^s(\mathbb{T}^d\times \mathbb{R}^d)}^\alpha} & \leqslant \frac{\norm{f_{\varepsilon} - f^0}_{L^2([0,t_{\varepsilon}]\times\mathbb{T}^d\times\mathbb{R}^d)}}{\norm{\jap{v}^m (f_{\varepsilon}|_{t=0} - f^0)}_{H^s(\mathbb{T}^d\times \mathbb{R}^d)}^\alpha}                                    \\
                                                                                                                                                                                            & \lesssim_{f^0} \sqrt{t_{\varepsilon}} + \frac{\norm{\nabla_x V^{\ext}_{\varepsilon}}_{L^2_{t_{\varepsilon}}(L^1_{t}L^2_{x})}}{\norm{\jap{v}^m (f_{\varepsilon}|_{t=0} - f^0)}_{H^s(\mathbb{T}^d\times \mathbb{R}^d)}^\alpha}                   \\
                                                                                                                                                                                            & \lesssim_{f^0} \sqrt{t_{\varepsilon}} + \frac{\sqrt{t_{\varepsilon}}\norm{\nabla_x V^{\ext}_{\varepsilon} - 0}_{L^1_{t_{\varepsilon}}L^2_{x}}}{\norm{\jap{v}^m (f_{\varepsilon}|_{t=0} - f^0)}_{H^s(\mathbb{T}^d\times \mathbb{R}^d)}^\alpha}.
\end{align*}
Since \(f^0\) only depends on \(v\), its associated exterior potential is zero. Finally, letting \(\varepsilon\to 0\) yields the desired result.

\begin{appendices}
    \section{Propagation of an infinite number of moments}\label{appendix_moments}

    The uniform moment bounds hypothesis \eqref{hyp_thmITA_moments_bound} in \Cref{thm_injectivity_time_analytic} on the solution of the Vlasov--Poisson equation can be guaranteed by assuming the same condition on the initial datum. More precisely, it follows from a result in \cite[Proposition~3.2]{MR3842911} that we recall here for convenience and adapted to our setting.

    Define the weighted Sobolev norms, for \(n\in\mathbb{N}\), \(r\in\mathbb{R}\),
    \[
        \norm{f}_{\mathcal{H}^n_r}^2\colonequals\sum_{\abs{\alpha}+\abs{\beta}\leqslant n}\int_{\mathbb{T}^d}\int_{\mathbb{R}^d}\jap{v}^{2r}\abs{\partial_x^\alpha\partial_v^\beta f(x,v)}^2\dd v\dd x.
    \]
    \begin{lemma}
        Let \(n > d + 1\) and \(r > \frac{d}{2}\). Assume that \(f^0 \in \mathcal{H}^n_r\). Then, there exists \(T > 0\) such that there is a unique solution \(f\) with initial datum \(f^0\) to \eqref{eq_vlasov-poisson} such that \(f \in C([0,T];~\mathcal{H}^n_r)\).
    \end{lemma}

    It is not explicit in \cite{MR3842911} that the existence time \(T\) is independent of \(n\) and \(r\), which remains to be proved to get the propagation of an infinite number of moments.

    \begin{proposition}\label{prop_propagation_moments}
        Assume that \(f^0 \in \bigcap_{n, r}\mathcal{H}^n_r\). Then, there exists \(T > 0\) such that there is a unique solution \(f\) with initial datum \(f^0\) to \eqref{eq_vlasov-poisson} such that \(f \in C([0,T];~\bigcap_{n, r}\mathcal{H}^n_r)\). Moreover, we have the pointwise bound
        \[
            \forall k\in\mathbb{N},\qquad \abs{v}^{k+1}\abs{\nabla_x^k f(t,x,v)} \leqslant C_k(v),
        \]
        where each \(C_k(\cdot)\in L^1(\mathbb{R}^d_v)\) and the bounds hold uniformly in \((t,x)\).
    \end{proposition}

    \begin{proof}
        Let \(n, r>0\) large enough such that \Cref{prop_propagation_moments} applies. Take \(r'\geqslant r\). Let \(f^0\in\mathcal{H}^{n+1}_{r'}\subset\mathcal{H}^{n}_{r}\). Then, there exists \(T>0\) such that the Vlasov--Poisson equation \eqref{eq_vlasov-poisson} admits a unique solution \(f\in C([0,T];\mathcal{H}^{n}_{r})\) with initial datum \(f^0\). We proceed as before to obtain an a priori estimate on \(\norm{f(t)}_{\mathcal{H}^{n+1}_{r'}}\),
        \begin{align*}
            \partial_t\norm{f(t)}_{\mathcal{H}^{n+1}_{r'}} & \lesssim_d (1 + \norm{\nabla_xV(t)}_{H^{n+1}}+\norm{\nabla_xV^{\ext}(t)}_{H^{n+1}})\norm{f(t)}_{\mathcal{H}^{n+1}_{r'}} \\
                                                           & \lesssim_d (C_n + \norm{\rho(t)}_{H^{n}})\norm{f(t)}_{\mathcal{H}^{n+1}_{r'}}                                           \\
                                                           & \lesssim_{d, r} (C_n + \norm{f(t)}_{\mathcal{H}^{n}_{r}})\norm{f(t)}_{\mathcal{H}^{n+1}_{r'}},
        \end{align*}
        where we used elliptic regularity and Sobolev injection. By Grönwall's inequality, we deduce that
        \[
            \norm{f(t)}_{\mathcal{H}^{n+1}_{r'}}\leqslant\norm{f^0}_{\mathcal{H}^{n+1}_{r'}} e^{C_{d, r}\int_0^t(C_n+\norm{f(s)}_{\mathcal{H}^{n}_{r}})\dd s}\leqslant\norm{f^0}_{\mathcal{H}^{n+1}_{r'}}e^{C_{d, r}(C_n+\norm{f}_{L^{\infty}_T\mathcal{H}^{n}_{r}})T}<\infty.
        \]
        Hence, we propagate one additional derivative and the weight \(r'\) also up to time \(T\). By induction, that shows that the existence time \(T\) is independent of \(n\) and \(r\). Finally, we use the Sobolev injection \(H^{d/2+}\hookrightarrow L^{\infty}\) to get the pointwise bound.
    \end{proof}
\end{appendices}


\begin{thebibliography}{RPvL15}

    \bibitem[Bar20]{MR4093619}
    Aymeric Baradat.
    \newblock Nonlinear instability in {V}lasov type equations around rough
    velocity profiles.
    \newblock {\em Ann. Inst. H. Poincar\'e{} C Anal. Non Lin\'eaire},
    37(3):489--547, 2020.

    \bibitem[BD85]{MR794002}
    C.~Bardos and P.~Degond.
    \newblock Global existence for the {V}lasov-{P}oisson equation in {\(3\)} space
    variables with small initial data.
    \newblock {\em Ann. Inst. H. Poincar\'e{} Anal. Non Lin\'eaire}, 2(2):101--118,
    1985.

    \bibitem[Ben89]{MR1056129}
    Sa\"id Benachour.
    \newblock Analyticit\'e{} des solutions des \'equations de {V}lassov-{P}oisson.
    \newblock {\em Ann. Scuola Norm. Sup. Pisa Cl. Sci. (4)}, 16(1):83--104, 1989.

    \bibitem[BF68]{BERTRAND196868}
    P.~Bertrand and M.R. Feix.
    \newblock Non linear electron plasma oscillation: the “water bag model”.
    \newblock {\em Physics Letters A}, 28(1):68--69, 1968.

    \bibitem[BMM16]{MR3489904}
    Jacob Bedrossian, Nader Masmoudi, and Cl\'ement Mouhot.
    \newblock Landau damping: paraproducts and {G}evrey regularity.
    \newblock {\em Ann. PDE}, 2(1):Art. 4, 71, 2016.

    \bibitem[BN12]{MR3026566}
    Claude Bardos and Anne Nouri.
    \newblock A {V}lasov equation with {D}irac potential used in fusion plasmas.
    \newblock {\em J. Math. Phys.}, 53(11):115621, 16, 2012.

    \bibitem[BR91]{MR1126425}
    J\"urgen Batt and Gerhard Rein.
    \newblock Global classical solutions of the periodic {V}lasov-{P}oisson system
    in three dimensions.
    \newblock {\em C. R. Acad. Sci. Paris S\'er. I Math.}, 313(6):411--416, 1991.

    \bibitem[Bre89]{Brenier1989}
    Yann Brenier.
    \newblock A vlasov-poisson type formulation of the euler equations for perfect
    incompressible fluids.
    \newblock {\em Rapport de recherche INRIA}, 1989.

    \bibitem[Bre00]{MR1748352}
    Y.~Brenier.
    \newblock Convergence of the {V}lasov-{P}oisson system to the incompressible
        {E}uler equations.
    \newblock {\em Comm. Partial Differential Equations}, 25(3-4):737--754, 2000.

    \bibitem[Caf90]{MR1027897}
    Russel~E. Caflisch.
    \newblock A simplified version of the abstract {C}auchy-{K}owalewski theorem
    with weak singularities.
    \newblock {\em Bull. Amer. Math. Soc. (N.S.)}, 23(2):495--500, 1990.

    \bibitem[FLLS16]{PhysRevA.93.062510}
    S\o{}ren Fournais, Jonas Lampart, Mathieu Lewin, and Thomas~\O{}stergaard
    S\o{}rensen.
    \newblock Coulomb potentials and taylor expansions in time-dependent
    density-functional theory.
    \newblock {\em Phys. Rev. A}, 93:062510, 6 2016.

    \bibitem[HK64]{PhysRev.136.B864}
    P.~Hohenberg and W.~Kohn.
    \newblock Inhomogeneous electron gas.
    \newblock {\em Phys. Rev.}, 136:B864--B871, Nov 1964.

    \bibitem[HK19]{MR3842911}
    Daniel Han-Kwan.
    \newblock On propagation of higher space regularity for nonlinear {V}lasov
    equations.
    \newblock {\em Anal. PDE}, 12(1):189--244, 2019.

    \bibitem[HKH15]{MR3306612}
    Daniel Han-Kwan and Maxime Hauray.
    \newblock Stability issues in the quasineutral limit of the one-dimensional
    {V}lasov-{P}oisson equation.
    \newblock {\em Comm. Math. Phys.}, 334(2):1101--1152, 2015.

    \bibitem[HKN16]{MR3509003}
    Daniel Han-Kwan and Toan~T. Nguyen.
    \newblock Ill-posedness of the hydrostatic {E}uler and singular {V}lasov
    equations.
    \newblock {\em Arch. Ration. Mech. Anal.}, 221(3):1317--1344, 2016.

    \bibitem[HKR16]{MR3592362}
    Daniel Han-Kwan and Fr\'ed\'eric Rousset.
    \newblock Quasineutral limit for {V}lasov-{P}oisson with {P}enrose stable data.
    \newblock {\em Ann. Sci. \'Ec. Norm. Sup\'er. (4)}, 49(6):1445--1495, 2016.

    \bibitem[HP90]{PhysRevLett.64.3019}
    Gregory~W. Hammett and Francis~W. Perkins.
    \newblock Fluid moment models for landau damping with application to the
    ion-temperature-gradient instability.
    \newblock {\em Phys. Rev. Lett.}, 64:3019--3022, Jun 1990.

    \bibitem[KS65]{PhysRev.140.A1133}
    W.~Kohn and L.~J. Sham.
    \newblock Self-consistent equations including exchange and correlation effects.
    \newblock {\em Phys. Rev.}, 140:A1133--A1138, Nov 1965.

    \bibitem[LP91]{MR1115549}
    P.-L. Lions and B.~Perthame.
    \newblock Propagation of moments and regularity for the {\(3\)}-dimensional
    {V}lasov-{P}oisson system.
    \newblock {\em Invent. Math.}, 105(2):415--430, 1991.

    \bibitem[Man20]{Manfredi_2020}
    Giovanni Manfredi.
    \newblock Density functional theory for collisionless plasmas – equivalence
    of fluid and kinetic approaches.
    \newblock {\em Journal of Plasma Physics}, 86(2):825860201, 2020.

    \bibitem[MV11]{MR2863910}
    Cl\'ement Mouhot and C\'edric Villani.
    \newblock On {L}andau damping.
    \newblock {\em Acta Math.}, 207(1):29--201, 2011.

    \bibitem[Nir72]{MR322321}
    L.~Nirenberg.
    \newblock An abstract form of the nonlinear {C}auchy-{K}owalewski theorem.
    \newblock {\em J. Differential Geometry}, 6:561--576, 1972.

    \bibitem[Pen60]{10.1063/1.1706024}
    Oliver Penrose.
    \newblock Electrostatic instabilities of a uniform non‐maxwellian plasma.
    \newblock {\em The Physics of Fluids}, 3(2):258--265, 03 1960.

    \bibitem[Pfa92]{MR1165424}
    K.~Pfaffelmoser.
    \newblock Global classical solutions of the {V}lasov-{P}oisson system in three
    dimensions for general initial data.
    \newblock {\em J. Differential Equations}, 95(2):281--303, 1992.

    \bibitem[RG84]{PhysRevLett.52.997}
    Erich Runge and E.~K.~U. Gross.
    \newblock Density-functional theory for time-dependent systems.
    \newblock {\em Phys. Rev. Lett.}, 52:997--1000, 3 1984.

    \bibitem[RPvL15]{Ruggenthaler_2015}
    Michael Ruggenthaler, Markus Penz, and Robert van Leeuwen.
    \newblock Existence, uniqueness, and construction of the density-potential
    mapping in time-dependent density-functional theory.
    \newblock {\em Journal of Physics: Condensed Matter}, 27(20):203202, 4 2015.

    \bibitem[Sch91]{MR1132787}
    Jack Schaeffer.
    \newblock Global existence of smooth solutions to the {V}lasov-{P}oisson system
    in three dimensions.
    \newblock {\em Comm. Partial Differential Equations}, 16(8-9):1313--1335, 1991.

    \bibitem[Tay23]{MR4703940}
    Michael~E. Taylor.
    \newblock {\em Partial differential equations {I}. {B}asic theory}, volume 115
    of {\em Applied Mathematical Sciences}.
    \newblock Springer, Cham, third edition, 2023.

    \bibitem[VR21]{MR4266238}
    Renato Velozo~Ruiz.
    \newblock Gevrey regularity for the {V}lasov-{P}oisson system.
    \newblock {\em Ann. Inst. H. Poincar\'e{} C Anal. Non Lin\'eaire},
    38(4):1145--1165, 2021.

\end{thebibliography}
\end{document}